\documentclass[12pt,reqno,letterpaper]{amsart}
\usepackage{mathrsfs}

\usepackage[T1]{fontenc}
\usepackage{enumitem}
\usepackage{graphicx}
\graphicspath{ {./images/} }
\usepackage{esint}
\usepackage{hyperref}
\hypersetup{
  colorlinks   = true,
  citecolor    = blue,
  linkcolor    = blue
}

\usepackage{amsmath,amsfonts,amsthm,color,tikz, comment, xcolor, xfrac,amssymb}

\usepackage[normalem]{ulem}
\usepackage{mathrsfs}

\usepackage{pdfsync}
\usepackage[font={scriptsize}]{caption}
\usepackage{environ}

\usepackage[left=1in, right=1in, top=1in,bottom=1in]{geometry}
\numberwithin{equation}{section}

\def\XXint#1#2#3{{\setbox0=\hbox{$#1{#2#3}{\int}$ }
\vcenter{\hbox{$#2#3$ }}\kern-.6\wd0}}

\usepackage{fix-cm}

\newcommand{\C}{\mathbb C}

\newcommand{\R}{\mathbb R}
\newcommand{\N}{\mathbb N}

\newcommand{\T}{\mathbb T}
\newcommand{\Z}{\mathbb Z}

\newtheorem{theorem}{Theorem}[section]

\newtheorem{definition}[theorem]{Definition}

\newtheorem{lemma}[theorem]{Lemma}

\newtheorem{proposition}[theorem]{Proposition}

\theoremstyle{remark}
\newtheorem{remark}[theorem]{Remark}

\theoremstyle{remark}

\newcommand{\bean}{\begin{eqnarray*}}
\newcommand{\eean}{\end{eqnarray*}}
\newcommand{\ben}{\begin{enumerate}}
\newcommand{\een}{\end{enumerate}}
\newcommand{\beq}{\begin{equation}}
\newcommand{\eeq}{\end{equation}}

\allowdisplaybreaks

\begin{document}

\author{Nicholas Gismondi and Alexandru F. Radu}

\title{Intermittent solutions of the stationary 2D surface quasi-geostrophic equation}

\begin{abstract}
    In this paper we construct non-trivial solutions to the stationary dissipative surface quasi-geostrophic equation on the two dimensional torus which lie strictly below the critical regularity threshold of $\dot{H}^{-1/2}(\T^2)$. Specifically, for any $\alpha < 1/2$ and any dissipation exponent $0 < \gamma \leq 2$ we construct non-trivial solutions such that
    $$
        u,\theta \in \dot{B}^{\alpha-1}_{\infty,\infty}(\T^2) \cap \dot{B}^{\alpha-1}_{2,2}(\T^2).
    $$
    Due to the fact our solutions do not lie in $\dot{H}^{-1/2}(\T^2)$, this requires reinterpreting the notion of a solution. This leads us to formulate the notion of a weak paraproduct solution for the stationary SQG equation. The main new ingredient is the incorporation of intermittency into the construction of the solutions. This allows us to demonstrate non-trivial integrability results for certain fractional derivatives of our solutions. In particular, for highly intermittent solutions, we are able to conclude for every $1 \leq p < 4/3$ we can construct $u$ and $\theta$ lying in $L^p(\T^2)$.
\end{abstract}

\maketitle

\section{Introduction}

\subsection{Motivation and Background}
In this paper we consider the stationary dissipative surface quasi-geostrophic (SQG) equation
\begin{equation}\label{eq:SQG}
    \begin{cases}
        \operatorname{div}(\theta u) + \Lambda^{\gamma}\theta = 0 \\
        u = \Lambda^{-1} \nabla^{\perp} \theta.
    \end{cases}
\end{equation}
for $\theta:\T^2 \to \R$ and $u:\T^2 \to \R^2$ mean-zero and $0 < \gamma \leq 2$. The non-stationary dissipative SQG equation, or just SQG equation for simplicity, is given by
\begin{equation}\label{eq:SQG_not_stat}
    \begin{cases}
        \partial_t \theta + \operatorname{div}(\theta u) + \Lambda^{\gamma}\theta = 0 \\
        u = \Lambda^{-1} \nabla^{\perp} \theta
    \end{cases}
\end{equation}
where now $\theta:[0,T) \times \T^2 \to \R$ and $u:[0,T) \times \T^2 \to \R^2$ and $T > 0$. The \textit{inviscid} SQG equation is simply ~\eqref{eq:SQG_not_stat} with the $\Lambda^{\gamma} \theta$ term omitted. By convention, when $\gamma = 0$ this refers to the inviscid SQG equation. From a geophysical fluid dynamics point of view, the SQG equation is an important model which describes the potential temperature at the surface of a rotating stratified fluid. Examples of such fluids include the surface layer of both the atmosphere and the ocean. See \cite{Salmon} for more discussion on the physical relevance and application of the equation.\\
From a mathematical point of view, the inviscid SQG equation first was proposed as an object of study by Constantin, Majda, and Tabak \cite{CMT}. There are a variety of reasons for this, but perhaps the simplest of these is that $\nabla^{\perp} \theta$ obeys the same evolution equation as the vorticity in the 3D Euler equations. That is, if $\omega$ denotes the vorticity of the velocity $u$ in the 3D Euler equation then it is well known that
$$
    \frac{D\omega}{Dt} = (\omega \cdot \nabla)u
$$
where $D/Dt = \partial_t + u \cdot \nabla$ denotes the material derivative. While, if $\theta$ and $u$ now solve ~\eqref{eq:SQG_not_stat} without dissipation, then
$$
    \frac{D\nabla^{\perp}\theta}{Dt} = (\nabla^{\perp}\theta \cdot \nabla)u.
$$
This suggests that ~\eqref{eq:SQG_not_stat} without dissipation may serve as a model equation for 3D Euler. See \cite{CMT} for more discussion on the analytic and geometric properties of solutions shared by both 2D SQG and 3D Euler.\\
On the topic of non-uniqueness of solutions to ~\eqref{eq:SQG_not_stat}, Buckmaster, Shkoller, and Vicol \cite{BSV} demonstrate that for every $1/2 < \beta < 4/5$, every $0 < \gamma < 2 - \beta$, every $\sigma < \beta/(2-\beta)$, and for every $\mathcal{H}:\mathbb{R} \to \mathbb{R}^+$ smooth and of compact support there exists non-trivial solutions which satisfy $\Lambda^{-1}\theta \in C_t^{\sigma} C_x^\beta$ and
$$
    \int_{\T^2} |\Lambda^{-1}\theta(t,x)|^2\, dx = \mathcal{H}(t).
$$
The proof introduces an auxiliary equation known as the \textit{relaxed SQG momentum equation} (See ~\eqref{eq:relaxed}). The reason for this is the so-called odd multiplier obstruction. To elaborate on this, in the standard convex integration methodology, first introduced by De Lellis and Sz\'{e}kelyhidi in \cite{DLS13} and \cite{DLS09}, one hopes to utilize interactions between high frequency terms which result from the non-linearity to produce low frequency terms which cancel errors. However, if one utilizes this methodology with ~\eqref{eq:SQG_not_stat} (or ~\eqref{eq:SQG}),
then the high frequency terms in essence perfectly cancel, leaving behind the errors. Heuristically, this is due to the odd Fourier multiplier which relates $u$ and $\theta$. There has been much work done on active scalar equations when the relationship between $u$ and $\theta$ is given by a Fourier multiplier which is not odd; see for instance \cite{DW}, \cite{IV}, and \cite{Shvydkoy}. In the setting of the SQG equation though, the relaxed momentum equation considers $u$ and $v$, where $u = \Lambda v$. The Fourier multiplier of $\Lambda$ is $2\pi|\cdot|$ (see Definition \ref{def:laplace}), which of course is even, and thus the odd multiplier obstruction is completely side-stepped.

\noindent
In connection with this work on the regularity threshold required for energy to be conserved, we note the recent resolution of the Onsager conjecture for the inviscid SQG equation. Originally formulated for the 3D Euler equation by Onsager \cite{Onsager}, the conjecture identifies the critical regularity threshold for conservation of energy of weak solutions. Isett and Vicol \cite{IV} first established the rigid side of the conjecture by showing that the energy is conserved for all solutions which satisfy $\theta \in L^3_{t,x}$, while Dai, Giri, and Radu \cite{DGR} and independently Isett and Looi \cite{IL}, proved there exist $\Lambda^{-1} \theta \in C(\R,C^{\alpha}(\T^2))$ for $1/2 \leq \alpha < 1$ which do not necessarily conserve energy using a two step process, first involving Newton iteration which was introduced in \cite{GR}, followed by a convex integration scheme.

\noindent
Finally we mention the work of Cheng, Kwon, and Li \cite{CKL}. Here they consider ~\eqref{eq:SQG}, and prove that for $0 < \gamma < 3/2$ there exist nontrivial solutions in $\dot{H}^{-1/2}$ which satisfy $\Lambda^{-1}\theta \in C^{\alpha}(\T^2)$ for $1/2 \leq \alpha \leq 1/2 + \min(1/6,3/2-\gamma)$. In the proof they introduce the function $f = \Lambda^{-1}\theta$. The relation between $u$ and $f$ is given by $u = \nabla^{\perp}f$, which still has an odd Fourier multiplier, and as expected, the leading order terms from the interactions between their high frequency terms perfectly cancel. Interestingly however, from the highest order surviving term they are able to extract a non-trivial non-oscillatory term to give the desired cancellation of the errors; this stands in stark contrast to the Euler equation. See \cite{CKL} for more details.
\noindent
The aim of our work is to demonstrate the existence of nontrivial weak solutions of ~\eqref{eq:SQG} for any dissipation exponent $0 < \gamma \leq 2$ and belonging to $\dot{B}^{\alpha-1}_{\infty,\infty}(\T^2) \cap \dot{B}^{\alpha-1}_{2,2}(\T^2)$ for some fixed $\alpha <1/2$. It is well established how to define a weak solution to ~\eqref{eq:SQG} when $\theta \in \dot{H}^{-1/2}(\T^2)$; for instance, from \cite{CKL} we say this is a solution if for every $\psi \in C^{\infty}(\T^2)$ we have
\begin{equation}\label{eq:CKL_soln}
    \frac{1}{2} \int_{\T^2} \left(\Lambda^{-1/2}\theta\right) \Lambda^{1/2}\left([R^{\perp},\nabla \psi]\theta\right) = -\int_{\T^2} \left(\Lambda^{-1/2}\theta\right) \Lambda^{\gamma+1/2}\psi
\end{equation}
where $[R^{\perp},\nabla \psi]\theta = -[R_2,\partial_1\psi]\theta + [R_1,\partial_2\psi]\theta$, $R_j$ are the $j^{th}$ Riesz transforms for $j=1,2$, and $[A,B] = AB - BA$ denotes the standard commutator.  Since
$$
    \Vert [R_j,\psi]\theta \Vert_{\dot{H}^{-1/2}} \lesssim \Vert \psi \Vert_{\dot{H}^3} \Vert \theta \Vert_{\dot{H}^{-1/2}},
$$
the integral in ~\eqref{eq:CKL_soln} is well defined. In our situation though, since $\theta \not \in \dot{H}^{-1/2}(\T^2)$, the definition provided by ~\eqref{eq:CKL_soln} breaks down. Instead, we draw inspiration from \cite{ABGN} and use paraproducts to define $\mathbb{P}_{\not=0}(\theta u)$ as an element of $\dot{H}^{-5}(\T^2)$ and then use dual pairings to define the notion of a weak solution. See Definitions \ref{def:paras} and \ref{def:weak_para_soln}. We also note briefly that weak solutions to ~\eqref{eq:SQG_not_stat} for functions $\theta$ which for fixed values of $t$ lie in $\dot{H}^{-1/2}(\T^2)$ can also be defined; see for instance \cite{BSV} and \cite{FM}.

\noindent
The main new contribution of this paper is incorporating intermittency into the construction of the solutions. Intermittency has played a key role in several previous convex integration constructions; see for instance \cite{ABGN}, \cite{BBV}, \cite{BCV}, \cite{BMNV}, \cite{BV19}, \cite{CheskidovLuo}, \cite{CL2}, \cite{DS17}, \cite{Luo}, \cite{MS}, \cite{NV} and references therein. In particular, we make use of a new class of intermittent building blocks which we call \textit{intermittent blobs}. In contrast to standard intermittent building blocks used in previous convex integration constructions like Mikado flows and intermittent jets, the perpendicular gradient of the modulated intermittent blobs are not stationary solutions of any PDE (Euler, SQG, etc.) nor are they anisotropic. Indeed, these intermittent blobs are built from a radially symmetric, compactly supported base function and they incorporate intermittency in two full spatial dimensions.

\noindent
Regarding the first potential pitfall, in standard Nash-style iteration arguments in fluid dynamics, it has always been required that the building blocks be either stationary solutions or approximately stationary solutions of some PDE, usually the Euler equation. The reason for this is in the error estimates, when the derivative may possibly land on a high frequency object, that high frequency object being a stationary solution to a certain PDE is essential to obtain the desired estimates. In our setting, however, the Reynolds stress is measured in $\dot{H}^{-4}(\T^2)$, which ensures that derivatives never need to land on the high-frequency building blocks at all. As a result, this property is unnecessary. This in turn means that one can use a fully intermittent building block at no cost. Moreover, the full intermittency enables $u$ and $\theta$ to converge in spaces that are strictly stronger than what a standard Nash iteration for SQG would permit, even though the Reynolds stress converges in a weaker space. There is a well known result of Chae and Constantin \cite{CC} which asserts that stationary, compactly supported solutions of Euler must satisfy that the mean of $u \otimes u$ is a multiple of the identity. This historically prevented fully intermittent building blocks from being used in convex integration constructions. But since our intermittent building blocks do not solve the Euler equation, this potential issue does not appear for us.

\noindent
The second notable deviation of our intermittent blobs from standard convex integration building blocks is the lack of anisotropy. Because of this, the intermittent blobs carry no intrinsic notion of direction. Ordinarily, such anisotropy is essential, as it encodes the directional structure needed to correct the Reynolds stress. In our construction, however, we compensate for this lack of built-in directionality by slightly weighting one coordinate direction in the velocity increment. This imbalance effectively reinstates the missing geometric information and allows our isotropic intermittent blobs to play the same role typically performed by anisotropic intermittent building blocks.\\
This incorporation of intermittency allows us to show for solutions which are close to the critical regularity threshold that
$$
    \Lambda^{-1/2}u,\Lambda^{-1/2}\theta \in L^p(\T^2)
$$
for all $1 \leq p < 2$. In contrast, for every $1 \leq p < 4/3$, we construct solutions which are highly intermittent and satisfy
$$
    u,\theta \in L^p(\T^2).
$$
Notice, when considering highly intermittent solutions, we have that $u$ and $\theta$ are integrable functions. To our knowledge, this gives the first construction of stationary dissipative SQG solutions which are not merely distributions. In comparison, Gomez-Serrano, Park, Shi, and Yao \cite{GPSY} are able to construct locally integrable solutions to inviscid ~\eqref{eq:SQG}, and existence of solutions to ~\eqref{eq:SQG_not_stat} for initial data lying in $L^p(\R^2)$ for $p \geq 4/3$ was studied by Marchand \cite[Theorem 1.2]{FM}.

\noindent
As expected, there is a necessary tradeoff between regularity and intermittency. One cannot have both a highly regular and highly intermittent solution to stationary SQG. This tradeoff is quantified precisely in ~\eqref{eq:beta_epsilon_est}.
In view of \cite{CKL}, it seems conceivable that for every $\gamma$ one should be able to construct solutions to ~\eqref{eq:SQG} which lie in $\dot{H}^{-1/2}(\T^2)$ and are intermittent, but to demonstrate this will almost certainly require a different approach.

\subsection{Main Result}

As mentioned already, our first task is to extend the definition of a weak solution to ~\eqref{eq:SQG} to $u$ and $\theta$ lying in some Sobolev space with negative regularity which is not contained in $\dot{H}^{-1/2}(\T^2)$. To do this, we need to make sense of the product $\theta u$. In general, there is not much that can be said about this product when both functions lie in a Sobolev space with negative regularity, but following \cite[Definition 1.1]{ABGN} we may offer the following definition of the mean free product:

\begin{definition}[\textbf{Paraproducts in $\dot H^s(\T^2)$}]\label{def:paras}
    Let $f,g$ be distributions, so that $\mathbb{P}_{2^j}(f), \mathbb{P}_{2^{j'}}(g)$ are well-defined for $j,j'\geq 0$ (see Definition \ref{def:projs}).  We say that $\mathbb{P}_{\not = 0}(fg)$ is well-defined as a paraproduct in $\dot{H}^s(\T^2)$ for some $s \in \R$ if
    $$
        \sum_{j,j' \geq 0} \left\Vert \mathbb{P}_{\not = 0}\left(\mathbb{P}_{2^j}(f) \mathbb{P}_{2^{j'}}(g)\right) \right\Vert_{\dot{H}^s} < \infty \, .
    $$
    Then we define
    $$
        \mathbb{P}_{\not = 0}(fg) = \sum_{j,j' \geq 0} \mathbb{P}_{\not = 0} \left(\mathbb{P}_{2^j}(f) \mathbb{P}_{2^{j'}}(g)\right) \, ,
    $$
    since the right-hand side is an absolutely summable series in $\dot H^s(\T^2)$.
\end{definition}
\noindent
With this definition, adapting \cite[Definition 1.2]{ABGN}, we may now define a weak solution to \eqref{eq:SQG} which is valid for $u$ and $\theta$ belonging to Sobolev spaces of arbitrary regularity.

\begin{definition}[\textbf{Weak paraproduct solutions to ~\eqref{eq:SQG}}] \label{def:weak_para_soln}
    If $\theta \in \dot{H}^s$, $s < 0$, and $u = \Lambda^{-1} \nabla^{\perp} \theta \in \dot{H}^s$, we say $u$ and $\theta$ form a weak paraproduct solution to the SQG equation if there is $s' \in \R$ such that $\mathbb{P}_{\not = 0}(\theta u)$ is well defined as a paraproduct in $\dot{H}^{s'}$ in the sense of the previous definition and
    $$
        \left\langle \theta, \Lambda^{\gamma} \phi \right\rangle_{\dot{H}^{s},\dot{H}^{-s}} - \left\langle \mathbb{P}_{\not= 0}(\theta u), \nabla \phi \right\rangle_{\dot{H}^{s'}, \dot{H}^{-s'}} = 0
    $$
    for all smooth $\phi$.
\end{definition}
\noindent
With this, our results are the following:

\begin{theorem}[\textbf{Non-trivial stationary solutions of 2D SQG equation}]\label{thm:main}
    Given any $0 < \gamma \leq 2$ and any $\alpha < 1/2$, we construct
    $$
        u,\theta \in \left(\dot{B}^{\alpha-1}_{\infty,\infty}(\T^2) \cap \dot{B}^{\alpha-1}_{2,2}(\T^2)\right) \setminus \{0\}
    $$
    such that
    \begin{enumerate}
        \item [(a)] $u = \Lambda^{-1} \nabla^{\perp} \theta$;
        \item [(b)] $\mathbb{P}_{\not =0} (\theta u)$ is well defined as a paraproduct in $\dot{H}^{-5}(\T^2)$;
        \item [(c)] $u$ and $\theta$ form a weak paraproduct solution to the stationary SQG equation in the sense of Definition \ref{def:weak_para_soln}.
    \end{enumerate}
    Moreover, when $-1/2 \leq \alpha < 1/2$, then
    $$
        \Lambda^{-1/2}u, \Lambda^{-1/2}\theta \in L^p(\T^2)
    $$
    for all $1 < p \leq 2$. And when $\alpha < -1/2$, for all $0 < \epsilon < 1$, $u$ and $\theta$ can be constructed such that
    $$
    \Lambda^{1/2-\epsilon'}u, \Lambda^{1/2-\epsilon'}\theta \in L^p(\T^2)
    $$
    for all $0 < \epsilon' < 1$ and $1 \leq p < 2$ satisfying
    $$
    p < \frac{2}{2-\epsilon'} \quad \text{and} \quad \epsilon < \frac{p\epsilon' - 2p + 2}{2-p}.
    $$
\end{theorem}

\begin{remark}
    In the scenario when $\alpha < -1/2$ in Theorem \ref{thm:main}, by taking $\epsilon'=1/2$ we see for every $p < 4/3$ and
    $$
    0 < \epsilon < \frac{4-3p}{2(2-p)}
    $$
    one may construct nontrivial $u,\theta \in L^p(\T^2)$. To the authors' knowledge, this is the first example of nontrivial solutions to the stationary dissipative SQG equation where it is proven the solutions are \textit{integrable functions} and not merely \textit{distributions}.
\end{remark}

\subsection{Outline of Paper} In Section \ref{section2} we recall the basics of Littlewood-Paley theory and function space theory as well as the construction of the intermittent blobs. We also prove some technical lemmas which will be useful in the ensuing arguments. In Section \ref{section3} we give a brief overview of the convex integration argument that is to follow and modify Definition \ref{def:weak_para_soln} to allow for weak paraproduct solutions to the relaxed SQG momentum equation. In Section \ref{section4} we state our inductive proposition, Proposition \ref{prop:ind}, and use it to prove Theorem \ref{thm:main}. The first portion of Section \ref{section5} is used to construct the increment $w_{q+1}$, which is then used to build the sequence of Nash iterates. The remaining part of Section \ref{section5} is dedicated to the proof of Proposition \ref{prop:ind}.

\subsection{Acknowledgments} N.G. was supported by the NSF through grant DMS-2400238. A.R. was partially supported by a grant of the Ministry of Research,
Innovation and Digitization, CCCDI - UEFISCDI, project number ROSUA-2024-0001, within PNCDI IV. Our thanks to Ataleshvara Bhargava for many stimulating conversations regarding this project and for suggesting areas of improvement in this paper. We also thank Filip Rossi for his help in preparing the manuscript.

\section{Background Theory and Technical Lemmas}\label{section2}
The following on Littlewood-Paley theory and Fourier multiplier operators can be found in \cite{BCD} and \cite{Grafakos}. The exact formulation of Definition \ref{def:projs} is based off \cite[Definition 2.1]{ABGN} and \cite[Equation 4.9]{BSV}.

\begin{definition}[\textbf{Littlewood-Paley projectors}]\label{def:projs}
    There exists $\varphi \colon \R^{2}\to[0,1]$, smooth, radially symmetric, and compactly supported in $\{6/7 \leq |\xi|\leq 2\}$ such that $\varphi(\xi) = 1$ on $\{1 \leq |\xi| \leq 12/7\}$,
    \begin{equation}
        \sum_{j\geq 0}\varphi(2^{-j}\xi)=1 \hspace{0.25cm} \text{ for all } \hspace{0.25cm} |\xi|\geq 1, \notag
    \end{equation}
    and $\operatorname{supp}\varphi_{j}\cap \operatorname{supp} \varphi_{j'}=\emptyset$ for all $|j-j'|\geq 2$, where $\varphi_j(\cdot) = \varphi(2^{-j}\cdot)$. We define the projection of a function $f$ on its $0$-mode by
    \begin{equation}
        \mathbb{P}_{0}f=\int_{\T^2} f, \notag
    \end{equation}
    and the projection on the $j^{\rm th}$ shell by
    \begin{equation}
        \mathbb{P}_{2^j}(f)(x)=\sum_{k\in \Z^2}\hat{f}(k)\varphi_{j}(k)e^{2\pi ik\cdot x} \, . \notag
    \end{equation}
    We also define $\mathbb{P}_{\neq 0}f:=(\operatorname{Id}-\mathbb{P}_{0})f$, and we also denote by $\hat{K}_{\simeq 1}$ a smooth radially symmetric bump function with support in the ball $\{\xi: |\xi| \leq \frac 18\} $, which also satisfies $\hat{K}_{\simeq 1}(\xi)=1$ on the smaller ball $\{\xi: |\xi| \leq \frac{1}{16}\}$. Then let $\mathbb{P}_{\leq \lambda}$ be the convolution operator that has $\hat{K}_{\simeq 1}\left(\frac{\xi}{\lambda}\right)$ as its Fourier multiplier.
\end{definition}

\begin{definition}[\textbf{$\dot H^s$ Sobolev spaces}] For $s \in \mathbb{R}$, we define
    $$\dot{H}^{s}(\mathbb{T}^{2})=\left\{ f: \sum_{j \in \Z^2 \setminus \{0\}}|j|^{2s}|\hat{f}(j)|^2<\infty\right\}  $$
    with the norm induced by the sum above.
\end{definition}
\begin{remark}
    For every $f\in \dot H^{s}$ for some $s\in \R$, we define the Fourier coefficients
    \[\hat{f}(k)=\int_{\T^{2}}e^{-2\pi ik\cdot x}f(x)dx, \hspace{0.25cm} \text{ where } \hspace{0.25cm} \T^{2}=[0,1]^{2},\]
    and so we can define $\mathbb{P}_{2^j}f$ for $j\geq 0$.  Note that each $\mathbb{P}_{2^j}f$ is smooth if $f\in \dot H^{s}$ irrespective of the value of $s\in \R$.
\end{remark}
\begin{remark}
    For $s > 0$ the following equivalent definition of the $\dot{H}^{-s}$ norm given by
    $$
        \Vert f \Vert_{\dot{H}^{-s}} = \sup_{\Vert \phi \Vert_{\dot{H}^s = 1}} \left|\left\langle f,\phi \right\rangle_{\dot{H}^{-s},\dot{H}^s}\right|
    $$
    is available. With this definition, it is easy to see that the $\dot{H}^{-s}$ norm can be pushed inside integrals, which we perform quite often in Section \ref{section_oscillation_error} without comment.
\end{remark}

\begin{remark}
    Henceforth when the homogeneous Sobolev space in question is clear, we will write $\langle \cdot,\cdot \rangle$ to denote the dual pairing.
\end{remark}

\begin{remark}
    Throughout we will also consider the $\dot{H}^{s}$ norm of matrix valued functions $A:\T^2 \to \R^{2 \times 2}$. By this we mean if $\Vert A \Vert_{op}$ denotes the operator norm of $A$ then
    $$
        \Vert A \Vert_{\dot{H}^{s}}^2 := \sum_{j \in \Z^2 \setminus \{0\}} |j|^{2s} \Vert \hat{A}(j) \Vert_{op}^2.
    $$
\end{remark}
\noindent
We now briefly recall the definitions of the spaces that we will utilize frequently throughout the paper.
\begin{definition}{\textbf{($\alpha$-H\"{o}lder spaces)}}
    For $0 \leq \alpha < 1$ and $f:\T^2 \to \R$, consider the periodic extension to $\R^2$ and define the seminorms
    $$
        [f]_{\alpha} = \sup_{x,y \in \R^2} \frac{|f(x) - f(y)|}{|x-y|^\alpha}
    $$
    and put
    $$
        C^\alpha(\T^2) = \left\{f \in C(\T^2) : [f]_{\alpha} < \infty\right\}.
    $$
    For $\beta \geq 1$, write $\beta = n + \alpha$ where $n \in \N$ and $0 \leq \alpha < 1$ and put
    $$
        \Vert f \Vert_{C^\beta} = \Vert f \Vert_{C^n} + [f]_{\alpha}
    $$
    and we say $f \in C^\beta(\T^2)$ if $\Vert f \Vert_{C^\beta} < \infty$.
\end{definition}

\begin{definition}{\textbf{(Homogeneous Besov spaces)}}
    For $\alpha \in \R$, $0<p,q\leq \infty$, and $f \in \mathcal{D}'(\T^2)$, define the homogeneous Besov space to be
    \[
        \dot{B}^{\alpha}_{p,q} = \left\{ f: \|f\|_{\dot{B}^{\alpha}_{p,q}} = \left\| 2^{j\alpha}\|\mathbb{P}_{2^j}f\|_{L^p(\T^2)}\right\|_{\ell^q_{j\geq 0}} <\infty \right\}.
    \]
\end{definition}
\begin{remark}
    Note for $\alpha > 0$ the space $\dot{B}^{\alpha}_{\infty,\infty}$ is equivalent to $\{f \in C^\alpha : \hat{f}(0) = 0\}$. In addition, the space $\dot{B}^{\alpha}_{2,2}$ is equivalent to the homogeneous Sobolev space $\dot{H}^\alpha$ for all $\alpha \in \R$.
\end{remark}

\noindent The $\alpha$-H\"{o}lder spaces in some sense are superseded by the homogeneous Besov spaces, since these spaces are well defined for negative regularity exponents. But in situations when the regularity is positive, it can be useful to instead consider the Besov spaces as being $\alpha$-H\"{o}lder spaces.

\begin{definition}[\textbf{Fractional Laplacian}]\label{def:laplace}
    For $u:\T^2 \to \C$ and $s \in \R$, define the fractional Laplacian operator $\Lambda^s = (-\Delta)^{s/2}$ by
    $$
        \left(\Lambda^su\right)^{\wedge}(j) = |2\pi j|^s \hat{u}(j) \quad \forall j \in \Z^2.
    $$
    When $s < 0$ we restrict the above definition to just those $j \not = 0$.
\end{definition}
\noindent In defining the Reynolds stress error, we utilize the fact that every mean-zero vector valued function on $\T^2$ can be written as the divergence of a $2 \times 2$ symmetric matrix. We then formally invert the divergence operator to recover the Reynolds stress. This formal operation is made precise in the following definition, which follows \cite{BV2020} and \cite[Definition 4.2]{DLS13}.
\begin{definition}[\textbf{Inverse divergence}]\label{def:R}
    If $u:\T^2 \to \R$ is a function which is mean-zero on $\T^2$, then we put
    \[
        (\mathcal{R}u)^{ij} := (\partial_i \Delta^{-1} u^j + \partial_j \Delta^{-1}u^i) - (\delta_{ij}+ \partial_i \partial_j \Delta^{-1})\operatorname{div}\Delta^{-1}u
    \]
    for $i,j \in \{1,2\}$.
\end{definition}
\noindent First note that by inspection it is clear that $\mathcal{R}u$ is a symmetric matrix, and one can also check that $\operatorname{div}(\mathcal{R}u) = u$ for $u$ mean-zero.

\noindent The following lemma is standard, and can be deduced as a consequence of the Poisson summation formula.

\begin{lemma}[\textbf{$L^p$ boundedness of projection operators}]\label{lem:proj}
    $\mathbb{P}_{\leq \lambda}$ is a bounded operator from $L^p$ to $L^p$ for $1 \leq p \leq \infty$ with operator norm independent of $\lambda$.
\end{lemma}
\noindent Geometric lemmas are standard tools in the convex integration literature; our formulation here follows \cite[Lemma 4.2]{BSV}.

\begin{lemma}[\textbf{Reconstruction of symmetric tensors}]\label{lem:geom}
    Let $B(I,\epsilon)$ be the ball of radius $\epsilon$ around the identity matrix in the space of $2\times2$ symmetric matrices. We can choose $\epsilon > 0$ such that there exists a finite set $\Omega \subset \mathbb{S}^{1}$ and smooth positive functions $\Gamma_k \in C^{\infty}(B(I,\epsilon))$ for $k \in \Omega$ such that the following hold:
    \begin{enumerate}
        \item $5\Omega \subseteq \mathbb{Z}^2$;
        \item If $k \in \Omega$, then $-k \in \Omega$ and $\Gamma_k = \Gamma_{-k}$;
        \item For all $R \in B(I,\epsilon)$ we have
              \[ R = \frac{1}{2} \sum_{k \in \Omega} (\Gamma_k(R))^2 k^{\perp} \otimes k^{\perp}; \]
        \item if $k,k' \in \Omega$ and $k \not = -k'$ then $|k + k'| \geq 1/2$.
    \end{enumerate}
\end{lemma}
\noindent We will choose $\Omega = \{\pm e_1, \pm (3/5,4/5), \pm (3/5,-4/5)\}$. To prove Lemma \ref{lem:geom}, one writes the identity as a linear combination of matrices of the form $k^{\perp} \otimes k^{\perp}$ for $k \in \Omega$ and then applies the inverse function theorem. See the discussion after \cite[Lemma 4.2]{BSV} for more details.\\
As mentioned in the introduction, the building blocks for our velocity increments will be a new class of intermittent functions which we refer to as intermittent blobs. Their construction is based on standard constructions given for intermittent jets but adapted to 2 dimensions; we refer the reader to \cite{BCV} and \cite{BV2020} for constructions of intermittent jets in three dimensions. One of the crucial differences between these intermittent jets and intermittent blobs is the lack of anisotropy in the blobs. This means that, geometrically, the blobs have no notion of direction. This lack of directionality will be handled in our velocity increment where we weigh the $k^\perp$ direction more than the $k$ direction. In this sense, all geometric information concerning direction is encoded in the increment.

\begin{lemma}[\textbf{Intermittent blobs}]\label{lem:boldW}
    Let $\lambda_{q+1}$ be a large integer and take $0 < \epsilon < 1$ such that $\lambda_{q+1}^{\epsilon}$ is also an integer. For each $k \in \Omega$ from Lemma \ref{lem:geom}, there exist smooth $\rho^k_{q+1}:\T^2 \to \R$ such that
    \begin{enumerate}
        \item\label{w:2} $\int_{\T^2} \rho_{q+1}^k = 0$;
        \item\label{w:3} $\Vert \nabla^\alpha \rho_{q+1}^k \Vert_{L^p(\T^2)} \lesssim \lambda_{q+1}^{(1-\epsilon)\left(1-\frac{2}{p}\right)} \lambda_{q+1}^{|\alpha|}$;
        \item\label{w:4} $\rho^k_{q+1}$ is $\left(\frac{\T}{\lambda_{q+1}^{\epsilon}}\right)^2$-periodic;
        \item\label{w:5} $\rho_{q+1}^k = \rho_{q+1}^{-k}$.
    \end{enumerate}
\end{lemma}

\begin{proof}
    Let $\phi:\R^2 \to \R$ be smooth, radially symmetric, supported in $B(0,1)$, mean-zero, and $L^2$ normalized. Define $\rho_{q+1}^k:\R^2 \to \R$ by
    \begin{equation}\label{eq:rho}
        \rho^k_{q+1}(x) = \sum_{n,m \in \Z} \lambda_{q+1}^{1-\epsilon} \phi\left(5\lambda_{q+1}k \cdot x +\lambda_{q+1}^{1-\epsilon}n, 5\lambda_{q+1}k^\perp \cdot x + \lambda_{q+1}^{1-\epsilon}m\right).
    \end{equation}
    We claim this function will have all of the desired properties. Items \ref{w:2} and \ref{w:4} are obvious. \ref{w:5} follows from utilizing the radial symmetry of $\phi$. For \ref{w:3}, our strategy will be to obtain the desired estimate for $p = 1$ and $p = \infty$, and then use interpolation to obtain the desired bound for all intermediate values of $p$. First, note that taking derivatives will cost a factor of $\lambda_{q+1}$, so without loss of generality we may assume that $|\alpha| = 0$. We start with the $L^\infty$ estimate. For this, note for $\lambda_{q+1}$ large enough, the terms inside the summation in ~\eqref{eq:rho} will be disjoint. Thus
    \begin{equation}\label{eq:rho_L^inf_est}
        \Vert \rho^k_{q+1} \Vert_{L^\infty(\T^2)} \lesssim \lambda_{q+1}^{1-\epsilon} \Vert \phi \Vert_{L^{\infty}(\R^2)} \simeq \lambda_{q+1}^{1-\epsilon}.
    \end{equation}
    Now for the $L^1$ estimate, each term in ~\eqref{eq:rho} is supported in a ball of measure $\lambda_{q+1}^{-2}$. Each of these supports is contained within a rectangle of area $\lambda_{q+1}^{-2\epsilon}$. Hence
    $$
        \left|\operatorname{supp}\left(\rho_{q+1}^k\right) \cap [0,1]^2 \right| \simeq \frac{\lambda_{q+1}^{-2}}{\lambda_{q+1}^{-2\epsilon}} = \lambda_{q+1}^{2(\epsilon-1)}.
    $$
    Utilizing this observation as well as ~\eqref{eq:rho_L^inf_est} we have
    \begin{equation}\label{eq:rho_L1_est}
        \Vert \rho_{q+1}^k \Vert_{L^1(\T^2)} \lesssim \Vert \rho_{q+1}^k \Vert_{L^\infty(\T^2)} \left|\operatorname{supp}\left(\rho_{q+1}^k\right) \cap [0,1]^2 \right| \lesssim \lambda_{q+1}^{\epsilon-1}.
    \end{equation}
    Interpolating (see \cite{Stein}, \cite{Triebel}) between the estimates provided by ~\eqref{eq:rho_L^inf_est} and ~\eqref{eq:rho_L1_est} gives \ref{w:3} and completes the proof.
\end{proof}

\noindent A useful tool for us will be to represent the intermittent blob using its Fourier series expansion. Thanks to the construction outlined above and the Poisson summation formula, such a representation is simple to obtain.

\begin{lemma}[\textbf{Fourier series representation of the intermittent blob}]\label{lem:fourier}
    The Fourier series representation of $\rho_{q+1}^{k}$ is
    \begin{equation}\label{eq:fourier_series}
        \rho_{q+1}^{k}(x) = \sum_{n,m \in \Z} \lambda_{q+1}^{\epsilon-1} \hat{\phi}\left(\lambda_{q+1}^{\epsilon-1}n,\lambda_{q+1}^{\epsilon-1}m\right) e^{2\pi i 5\lambda_{q+1}^{\epsilon}  (nk + mk^\perp) \cdot x}.
    \end{equation}
\end{lemma}
\begin{proof}
    Applying the Poisson summation formula to ~\eqref{eq:rho} gives ~\eqref{eq:fourier_series}.
\end{proof}

\noindent The following lemma will be useful in the standard high-low product estimates we will have to perform in Section \ref{section_oscillation_error}. Morally it allows us to treat the low frequency function in the high-low product as a constant function thus simplifying some computations.

\begin{lemma}[\textbf{Kato-Ponce-type product estimate}]\label{lem:pseudo_diff_cont}
    Fix $\alpha,\beta \in C^{\infty}(\T^2)$ with $\beta$ having zero mean. Then for $s \geq 0$ we have
    $$
        \Vert \alpha \beta \Vert_{\dot{H}^{-s}} \lesssim \Vert \alpha \Vert_{C^s} \Vert \beta \Vert_{\dot{H}^{-s}}.
    $$
\end{lemma}
\begin{proof}
    Fix $\phi \in \dot{H}^s$ with norm $1$ and mean-zero, and then write
    \begin{equation}\label{eq:sobolev_ineq}
        \left|\left\langle \alpha\beta, \phi \right\rangle\right| = \left|\langle \beta, \alpha\phi \rangle\right| \leq \Vert \beta \Vert_{\dot{H}^{-s}} \Vert \alpha \phi \Vert_{\dot{H}^s}.
    \end{equation}
    Using \cite[~Proposition 1]{BOZ} we have that
    \begin{equation}\label{eq:sobolev_ineq2}
        \Vert \alpha \phi \Vert_{\dot{H}^s} \lesssim \Vert \alpha \Vert_{L^{\infty}} \Vert \phi \Vert_{\dot{H}^s} + \Vert \Lambda^s \alpha \Vert_{L^{\infty}} \Vert \phi \Vert_{L^2} \lesssim \Vert \alpha \Vert_{C^s} \Vert \phi \Vert_{\dot{H}^s}.
    \end{equation}
    From ~\eqref{eq:sobolev_ineq} and ~\eqref{eq:sobolev_ineq2} we deduce
    $$
        \Vert \alpha \beta \Vert_{\dot{H}^{-s}} = \sup_{\Vert \phi \Vert_{\dot{H}^s}=1} \left|\left\langle \alpha\beta, \phi \right\rangle\right| \lesssim \Vert \alpha \Vert_{C^s} \Vert \beta \Vert_{\dot{H}^{-s}}
    $$
    which completes the proof.
\end{proof}

\section{Convex Integration Scheme}\label{section3}

Instead of working with ~\eqref{eq:SQG} directly, we instead choose to work with the relaxed momentum formulation:
\begin{equation}\label{eq:relaxed}
    \begin{cases}
        u \cdot \nabla v - (\nabla v)^T \cdot u  + \Lambda^{\gamma} v + \nabla p = 0 \\
        \operatorname{div}(v) = 0                                                    \\
        v = \Lambda^{-1} u
    \end{cases}
\end{equation}
This is first considered in \cite{BSV} as a solution to the odd multiplier obstruction. For us, not only does this help us avoid the odd multiplier obstruction, but this formulation is also more amenable to the convex integration scheme we wish to employ since a natural tensor product structure appears, allowing us to use the properties contained in Lemma \ref{lem:geom} and Lemma \ref{lem:boldW}. To give a high level overview of what is to follow, we start by assuming $(u_q,v_q,R_q,p_q)$ solve
\begin{equation}\label{eq:relax_Reynolds}
    \begin{cases}
        u_q \cdot \nabla v_q - (\nabla v_q)^T \cdot u_q + \Lambda^{\gamma} v_q + \nabla p_q = \operatorname{div}(R_q) \\
        \operatorname{div}(v_q) = 0                                                                                   \\
        v_q = \Lambda^{-1} u_q
    \end{cases}
\end{equation}
Then for carefully chosen $w_{q+1}$ if we set
$$
    v_{q+1} = v_q + w_{q+1}
$$
and
$$
    u_{q+1} = u_q + \Lambda w_{q+1}
$$
we will show that $v_q \to v \not = 0$ in both the $\dot{B}^\alpha_{\infty,\infty}$ and $\dot{H}^\alpha$ norm. From this, we will deduce that $u_q \to u$ in $\dot{H}^{\alpha-1}$, and since $u = \Lambda v$, then $u \not = 0$.\\
By applying the perpendicular divergence we have, using
$$
    \nabla^{\perp} \cdot \left(u_q \cdot \nabla v_q - (\nabla v_q)^T \cdot u_q\right) = \nabla^{\perp} \cdot \left[(\nabla^{\perp} \cdot v_q)u_q \right] = u_q \cdot \nabla(\nabla^{\perp} \cdot v_q) = \operatorname{div}((\nabla^{\perp} \cdot v_q) u_q)
$$
that
$$
    \operatorname{div}((\nabla^{\perp} \cdot v_q) u_q) + \Lambda^{\gamma}(\nabla^{\perp} \cdot v_q) = \nabla^{\perp} \cdot \operatorname{div}(R_q).
$$
This leads us to say that $(u_q,v_q,R_q)$ form a solution to the relaxed momentum equation with Reynolds stress in the weak paraproduct sense if
\begin{equation}\label{eq:weak_para_soln_relax}
    -\left\langle v_q^j, \left(\Lambda^{\gamma} \nabla^{\perp} \phi\right)^j \right\rangle_{\dot{H}^{s},\dot{H}^{-s}} + \left\langle \mathbb{P}_{\not= 0}((\nabla^{\perp} \cdot v_q) u_q^j), \partial_j \phi, \right\rangle_{\dot{H}^{s'}, \dot{H}^{-s'}} =  \langle R_q^{ij},(\nabla(\nabla^{\perp}\phi))^{ji} \rangle_{\dot{H}^{-4},\dot{H}^{4}}
\end{equation}
for some $s,s' < 0$. We will then show as $q \to \infty$ that the right hand side of ~\eqref{eq:weak_para_soln_relax} tends to $0$. This leads to a solution in the following sense:
\begin{definition}[\textbf{Weak paraproduct solutions to ~\eqref{eq:relaxed}}]\label{def:weak_para_soln_relax}
    If $v \in \dot{H}^s$, $s < 0$, $\operatorname{div}(v) = 0$ in the weak sense, and $u = \Lambda v \in \dot{H}^{s-1}$, we say $u$ and $v$ form a weak paraproduct solution to the relaxed SQG momentum equation if there is $s' \in \R$ such that $\mathbb{P}_{\not = 0}((\nabla^{\perp} \cdot v)u)$ is well defined as a paraproduct in $\dot{H}^{s'}$ in the sense of Definition \ref{def:paras} and
    $$
        -\left\langle v, \Lambda^{\gamma} \nabla^{\perp}\phi \right\rangle_{\dot{H}^{s},\dot{H}^{-s}} + \left\langle \mathbb{P}_{\not= 0}((\nabla^{\perp} \cdot v) u), \nabla \phi, \right\rangle_{\dot{H}^{s'}, \dot{H}^{-s'}} = 0
    $$
    for all smooth $\phi$.
\end{definition}
\noindent Upon making the substitution $\theta = -\nabla^{\perp} \cdot v$ we recover our weak paraproduct solution to the SQG equation in the sense of Definition \ref{def:weak_para_soln} so these definitions are equivalent in the sense that having a weak paraproduct solution to one will give a weak paraproduct solution to the other.

\section{Inductive Proposition and Proof of Theorem \ref{thm:main}}\label{section4}

\begin{proposition}[\textbf{Inductive Proposition}]\label{prop:ind}
    Take $\alpha < 1/2$ rational and at least as large as the regularity indicated in Theorem \ref{thm:main}. We make the following inductive assumptions about $(u_q,v_q,p_q,R_q)$:
    \begin{enumerate}
        \item\label{i:2} $u_q$ and $v_q$ have zero mean and are divergence free;
        \item\label{i:1} $(u_q,v_q,p_q,R_q)$ are smooth solutions of ~\eqref{eq:relax_Reynolds};
        \item\label{i:3}$\Vert R_q \Vert_{\dot{H}^{-4}} < 2^{-q}$;
        \item\label{i:4} For all $q' \leq q$ we have
              $$
                  \Vert v_{q'} - v_{q'-1} \Vert_{\dot{B}^{\alpha}_{\infty,\infty}} + \Vert v_{q'} - v_{q'-1} \Vert_{\dot{H}^{\alpha}} \lesssim 2^{-q'}
              $$
              where the implicit constant depends on $\alpha$ but not $q$ or $q'$;
        \item\label{i:5} There exists a constant $\delta > 0$ independent of $q$ such that $\Vert v_q \Vert_{L^1} > (1+2^{-q})\delta$.
        \item\label{i:6} Recalling the frequency projections $\mathbb{P}_{2^j}$ from Definition \ref{def:projs}, there exists a unique $j$ such that $\mathbb{P}_{2^j}(v_{q} - v_{q-1}) = v_{q} - v_{q-1}$.
        \item\label{i:7} There are $C_1,C_2 > 0$ independent of $q$ such that
              $$
                  \sum_{\substack{n,m \leq q\\ n \not  = m}} \Vert (\nabla^{\perp} \cdot (v_n - v_{n-1}) (u_m - u_{m-1}) \Vert_{\dot{H}^{-5}} < C_1 - 2^{-q}
              $$
              and
              $$
                  \sum_{n \leq q} \Vert (\nabla^{\perp} \cdot (v_n - v_{n-1}) (u_n - u_{n-1}) \Vert_{\dot{H}^{-5}} < C_2 - 2^{-q + 100}.
              $$
        \item\label{i:8} In the case when $\alpha \geq  -1/2$, for all $q' \leq q$ we have
              $$
                  \Vert \Lambda^{1/2}(v_{q'} - v_{q'-1})\Vert_{L^p(\T^2)} \lesssim 2^{-q'}
              $$
              for all $1 \leq p < 2$. Note the implicit constant depends on $p$ but not $q$ or $q'$. In the case when $\alpha < -1/2$, for every $0 < \epsilon < 1$ and every $0 < \epsilon' < 1$ and $1 \leq p < 2$ such that
              \begin{equation}\label{eq:conds}
                  p < \frac{2}{2-\epsilon'} \quad \text{and} \quad \epsilon < \frac{p\epsilon' - 2p + 2}{2 - p}
              \end{equation}
              we have
              $$
                  \left\Vert \Lambda^{3/2-\epsilon'}(v_{q'} - v_{q'-1})\right\Vert_{L^p(\T^2)} \lesssim 2^{-q'}.
              $$
    \end{enumerate}
\end{proposition}
\noindent With these assumptions, we prove the main result. The rest of the paper will be dedicated to the proof of Proposition \ref{prop:ind}.
\begin{proof}[Proof of Theorem~\ref{thm:main} using Proposition~\ref{prop:ind}] We start by verifying the base case of the induction hypothesis. Put
    $v_0 = A\sin\left(2\pi x_1\right)e_2$ for some constant $A$ to be chosen later and $p_0 = 0$. Then
    \begin{equation*}
        \begin{split}
            u_0 & = \Lambda v_0                                      \\
                & = A\Lambda\left(\sin(2\pi  e_1 \cdot x)\right)e_2  \\
                & = |2\pi e_1|A \sin\left(2\pi e_1 \cdot x\right)e_2 \\
                & = 2\pi A \sin\left(2\pi e_1 \cdot x\right)e_2.
        \end{split}
    \end{equation*}
    And similarly
    $$
        \Lambda^{\gamma} v_0 = (2\pi)^\gamma A \sin\left(2\pi e_1 \cdot x\right)e_2.
    $$
    Thus, if we set
    $$
        R_0 =
        \begin{bmatrix}
            \frac{\pi}{2}A^2 \cos(4\pi x_1)   & -(2\pi)^{\gamma -1}A \cos(2\pi x_1) \\
            -(2\pi)^{\gamma-1}A\cos(2\pi x_1) & 0                                   \\
        \end{bmatrix},
    $$
    then one can check that
    $$
        u_0 \cdot \nabla v_0 - (\nabla v_0)^T \cdot u_0 + \Lambda^{\gamma} v_0 = \operatorname{div}(R_0).
    $$
    And so, \ref{i:1} is satisfied. Item \ref{i:2} is obvious. Choosing $A > 0$ small enough ensures \ref{i:3}, \ref{i:4}, and \ref{i:7} (we set $v_{-1}=0$). Taking $\delta = (1/4)\Vert v_0 \Vert_{L^1} = (2\pi)^{-1}A$ gives \ref{i:5}. Clearly
    $$
        \mathbb{P}_{1}(v_0 - v_{-1}) = \mathbb{P}_1(v_0) = v_0 = v_0 - v_{-1},
    $$
    giving \ref{i:6}. By choosing $C_1$ and $C_2$ large enough, we ensure \ref{i:7}. So, the base case has been verified.\\
    Now assume that Proposition \ref{prop:ind} is satisfied for all $q \geq 0$. Set
    $$
        v = \lim_{q \to \infty} v_q = \lim_{q \to \infty} \sum_{0 \leq q' \leq q} (v_{q'} - v_{q'-1}).
    $$
    We will first show that this limit exists in $\dot{B}^{\alpha}_{\infty,\infty} \cap \dot{H}^{\alpha}$. Using \ref{i:4} we have
    $$
        \sum_{0 \leq q' \leq q} \left(\Vert v_{q'} - v_{q'-1} \Vert_{\dot{B}^{\alpha}_{\infty,\infty}} + \Vert v_{q'} - v_{q'-1} \Vert_{\dot{H}^{\alpha}}\right) \lesssim_{\epsilon} \sum_{q' \leq q} 2^{-q'} \lesssim_{\epsilon} 1.
    $$
    Hence $\{v_q\}$ is Cauchy in both $\dot{B}^{\alpha}_{\infty,\infty}$ and $\dot{H}^{\alpha}$, and thus converges to $v$ in both spaces.\\
    Since $u_q = \Lambda v_q$, we see that $\{u_q\} \subset \dot{H}^{\alpha-1}$ and thus
    $$
     \Vert u_q - u \Vert_{\dot{H}^{\alpha-1}}  = \Vert v_q - v \Vert_{\dot{H}^{\alpha}} \lesssim 2^{-q} \to 0
    $$
    Thus, $u_q \to u$ in both $\dot{B}^{\alpha-1}_{\infty,\infty}$ and $\dot{H}^{\alpha-1}$ norms. Then since $\theta = - \nabla^{\perp} \cdot v$, this means that $\theta \in \dot{B}^{\alpha-1}_{\infty,\infty}$ and $\dot{H}^{\alpha-1}$ as well.
    \noindent
    Now, using \ref{i:5} and the convergence in $L^1$ norm we have
    $$
        (1 + 2^{-q}) \delta < \Vert v_q \Vert_{L^1} \leq \Vert v_q - v \Vert_{L^1} + \Vert v \Vert_{L^1}.
    $$
    Sending $q \to \infty$ we see that $v \not =0$. As a consequence $u$ is not identically $0$. Finally since $u = \Lambda^{-1} \nabla^{\perp} \theta$, we have $\theta$ must also not be identically $0$.
    \noindent
    Next we verify that $v$ is weakly divergence free. By this we mean for every smooth $\psi:\T^2 \to \C$ we have that
    $$
        \int_{\T^2} v \cdot \nabla \psi = 0.
    $$
    And indeed, this is an easy consequence of the $L^1$ convergence of $v_q$ to $v$ and the fact that each $v_q$ is divergence free in the classical sense. So using this $L^1$ convergence as well as integration by parts we have
    $$
        \int_{\T^2} v \cdot \nabla \psi = \lim_{q \to \infty} \int_{\T^2} v_q \cdot \nabla \psi = \lim_{q \to \infty} \int_{\T^2}- \operatorname{div}(v_q) \psi = 0.
    $$
    Hence $v$ is weakly divergence free.\\
    Finally for $\phi$ smooth we analyze
    \begin{equation}\label{eq:integral_form_relax_Reynolds}
        -\int_{\T^2} v_q \cdot \Lambda^{\gamma} \nabla^{\perp} \phi + \int_{\T^2} (\nabla^{\perp} \cdot v_q)u_q \cdot \nabla\phi = \int_{\T^2} (\nabla^{\perp} \cdot \operatorname{div}(R_q)) \phi.
    \end{equation}
    By \ref{i:1}, this equality is valid. For the integral on the right hand side of ~\eqref{eq:integral_form_relax_Reynolds}, using integration by parts and \ref{i:3} we have
    \begin{equation}\label{eq:rhs}
        \begin{split}
            \left|\int_{\T^2} (\nabla^{\perp} \cdot \operatorname{div}(R_q)) \phi\right| = \left|\int_{\T^2} R_q : \nabla(\nabla^{\perp} \phi)\right| \lesssim \Vert R_q \Vert_{\dot{H}^{-4}} \Vert \nabla(\nabla^{\perp} \phi) \Vert_{\dot{H}^{4}} \lesssim 2^{-q}
        \end{split}
    \end{equation}
    where for $2 \times 2$ matrices $A$ and $B$ we have that
    $$
        A : B = \sum_{i,j=1,2} A_{ij} B_{ji}.
    $$
    Hence from ~\eqref{eq:rhs} the right hand side of ~\eqref{eq:integral_form_relax_Reynolds} converges to $0$ as we send $q \to \infty$. Now since $v_q \to v$ in $L^1$ we have
    \begin{equation}\label{eq:dissipation_term}
        \begin{split}
            \lim_{q \to \infty} -\int_{\T^2} v_q \cdot \Lambda^{\gamma} \nabla^{\perp} \phi = -\int_{\T^2} v \cdot \Lambda^{\gamma} \nabla^{\perp} \phi = -\left\langle v, \Lambda^{\gamma}\nabla^{\perp}\phi \right\rangle_{\dot{H}^{-\epsilon},\dot{H}^{\epsilon}}.
        \end{split}
    \end{equation}
    From \ref{i:6} and \ref{i:7}, $\mathbb{P}_{\not=0}((\nabla^{\perp} \cdot v)u)$ exists as a paraproduct in $\dot{H}^{-5}$.
    Thus
    \begin{equation}\label{eq:nonlinear_term}
        \begin{split}
            \lim_{q \to \infty} \int_{\T^2} (\nabla^{\perp} \cdot v_q)u_q \cdot \nabla\phi & = \lim_{q \to \infty} \int_{\T^2} \mathbb{P}_{\not=0}\left((\nabla^{\perp} \cdot v_q)u_q\right) \cdot \nabla\phi \\
                                                                                           & = \lim_{q \to \infty} \int_{\T^2} \sum_{\substack{n,m\leq q                                                      \\j,j' \geq 0}}  \mathbb{P}_{\not=0}\left(\mathbb{P}_{2^j}(\nabla^{\perp} \cdot (v_n-v_{n-1})) \mathbb{P}_{2^{j'}}(u_m-u_{m-1})\right) \cdot \nabla \phi\\
                                                                                           & = \lim_{q \to \infty} \left\langle \sum_{\substack{n,m\leq q                                                     \\j,j' \geq 0}} \mathbb{P}_{\not=0}\left(\mathbb{P}_{2^j}(\nabla^{\perp} \cdot (v_n-v_{n-1})) \mathbb{P}_{2^{j'}}(u_m-u_{m-1})\right), \nabla \phi\right\rangle\\
                                                                                           & := \left\langle \sum_{\substack{n,m                                                                              \\j,j' \geq 0}} \mathbb{P}_{\not=0}\left(\mathbb{P}_{2^j}(\nabla^{\perp} \cdot (v_n-v_{n-1})) \mathbb{P}_{2^{j'}}(u_m-u_{m-1})\right), \nabla \phi\right\rangle\\
                                                                                           & = \left\langle \mathbb{P}_{\not=0}((\nabla^{\perp} \cdot v)u), \nabla \phi\right\rangle_{\dot{H}^{-5},\dot{H}^5}.
        \end{split}
    \end{equation}
    Notice the fourth equality in ~\eqref{eq:nonlinear_term} is defined this way using Definition \ref{def:paras}. Hence combining ~\eqref{eq:rhs}, ~\eqref{eq:dissipation_term}, and ~\eqref{eq:nonlinear_term} we see that
    \begin{equation*}
        \begin{split}
            -\left\langle v, \Lambda^{\gamma}\nabla^{\perp}\phi \right\rangle_{\dot{H}^{-\epsilon},\dot{H}^{\epsilon}} + \left\langle \mathbb{P}_{\not=0}((\nabla^{\perp} \cdot v)u), \nabla \phi\right\rangle_{\dot{H}^{-5},\dot{H}^5} & = \lim_{q \to \infty} \int_{\T^2} \left(-v_q \cdot \Lambda^{\gamma} \nabla^{\perp} \phi +  (\nabla^{\perp} \cdot v_q)u_q \cdot \nabla\phi\right) \\
                                                                                                                                                                                                                                        & = 0.
        \end{split}
    \end{equation*}
    So, $u$ and $v$ form our desired nonzero relaxed paraproduct solutions to ~\eqref{eq:relaxed} in the sense of Definition \ref{def:weak_para_soln_relax}, and thus we may also recover a nonzero relaxed paraproduct solution to ~\eqref{eq:SQG} in the sense of Definition \ref{def:weak_para_soln}.\\
    Now assume $\alpha \geq -1/2$. From \ref{i:8}, $\{\Lambda^{1/2}v_q\}$ is a Cauchy sequence in $L^p$ for all $1 \leq p < 2$. Since $\Lambda^{1/2}v_q$ converges to $\Lambda^{1/2}v$ in the sense of distributions, we must have that $\Lambda^{1/2}v_q$ converges to $\Lambda^{1/2}v$ in $L^p$. Since $u = \Lambda v$, we see easily that $\Lambda^{-1/2}u \in L^p$. Now using that
    \begin{equation}\label{eq:equiv_Riesz}
        \Vert Rf \Vert_{L^p(\T^2)} \simeq \Vert f \Vert_{L^p(\T^2)}
    \end{equation}
    for $f \in L^p(\T^2)$, $1 < p < \infty$ and $f$ mean-zero and the fact $\T^2$ is a finite measure space, we deduce that $\Lambda^{-1/2}\theta \in L^p$ for all $1 \leq p < 2$. The equivalence of the above norms is easy to observe from the classical $L^p$ boundedness of the Riesz transforms \cite{Stein}, together with the reconstruction formula
    $$
        f = - R_1R_1f - R_2R_2f
    $$
    which is valid since $f$ is mean-zero.\\
    Now in the situation when $\alpha < -1/2$, fix $0 < \epsilon < 1$ and $\epsilon'$ and $p > 1$ satisfying ~\eqref{eq:conds}. By \ref{i:8} and similar reasoning as before, we have that $\Lambda^{3/2 -\epsilon'}v$ is an element of $L^p$. The same holds for $\Lambda^{1/2-\epsilon'}u$, and again applying ~\eqref{eq:equiv_Riesz} we obtain the same for $\Lambda^{1/2-\epsilon'}\theta$. In the case when $p = 1$, we simply consider $\tilde{p} > 1$, repeat the previous reasoning, and then use the fact $\T^2$ is a finite measure space to conclude $\Lambda^{1/2-\epsilon'}u,\Lambda^{1/2-\epsilon'}\theta \in L^1(\T^2)$.
\end{proof}
\section{Proof of Proposition~\ref{prop:ind}}\label{section5}

Throughout the rest of the paper, we will work with parameters $\alpha$, $\epsilon$, and $\lambda_{q+1}$. The parameter $\alpha < 1/2$ should be chosen to be rational and at least as large as the indicated regularity in Theorem \ref{thm:main}. The intermittency parameter $\epsilon$ will quantify the amount of intermittency present in the intermittent blobs used in the construction. We will require that $0 < \epsilon < 1$ is chosen to be rational and satisfy
\begin{equation}\label{eq:beta_epsilon_est}
    \alpha  + \frac{1}{2} < \epsilon.
\end{equation}
This choice of $\epsilon$ can always be made since $\alpha < 1/2$. The frequency parameter $\lambda_{q+1}$ is given by
$$
    \lambda_{q+1} = 2^{6(\alpha\epsilon)^{-1}f(q+1)}
$$
for some $f:\N \to \N$ such that
$$
    \frac{f(q+1)}{\alpha\epsilon} \in \N.
$$
for all $q$. We also assume that $\lambda_{q+1}^\alpha$, $\lambda_{q+1}^{\epsilon}$ and $\lambda_{q+1}^{1/2}$ are all integers. The parameter $\lambda_{q+1}$ will be assumed to be large enough (and thus also $f(q+1)$) to satisfy the conditions given by equations \eqref{est:lambda_ineq}, \eqref{eq:dis_est}, \eqref{eq:spt_don_1}, \eqref{eq:spt_cond_2}, \eqref{eq:Nash_error_est}, \eqref{eq:Q^2_est}, \eqref{eq:Q11_final_est}, \eqref{eq:Q12Q13_est}, \eqref{eq:Q14_finalest},
\eqref{eq:leftover_terms_est}, \eqref{eq:C^alpha_est_3}, \eqref{eq:H^alpha_est_3},
\eqref{eq:wq+1_est_delta}, \eqref{eq:ind_est_1},
\eqref{eq:ind_est_2},
\eqref{eq:pk2_est},
\eqref{eq:I1_est},
and
~\eqref{eq:I2_est}. For notational convenience, we set $\sigma_{q+1} = \frac{5}{8}\lambda_{q+1}$. The reason for this seemingly arbitrary choice will be made clear in Section \ref{Sec:freq_support}. Clearly we can require $\sigma_{q+1} \in \N$ since we assume $\lambda_{q+1}$ is at least $8$.

\begin{remark}
    Notice that equation ~\eqref{eq:beta_epsilon_est} demonstrates the necessary tradeoff between regularity and intermittency.
\end{remark}

\subsection{Construction of $w_{q+1}$ and various properties}

\begin{definition}[\textbf{Definition of $w_{q+1}$}]\label{def:wq+1}
    Define the increment $w_{q+1}$ by
    \begin{equation}\label{eq:wq+1}
        w_{q+1} = \nabla^{\perp}\left(\frac{1}{2\pi  \sigma_{q+1}} \sum_{k \in \Omega} \mathbb{P}_{\leq \lambda_{q+1}}\left(a_k(R_q(x)) \rho_{q+1}^{k}(x)\right) \left[e^{2\pi i \sigma_{q+1} (k + 2k^\perp) \cdot x} + e^{2\pi i \sigma_{q+1}(k - 2k^\perp) \cdot x}\right]\right)
    \end{equation}
    where $\rho_{q+1}^k$ is the intermittent blob from Lemma \ref{lem:boldW} and
    \begin{equation}\label{eq:ak}
        a_k(R_q) = \frac{1}{(\lambda_{q+1}\epsilon_q)^{1/2}}  \Gamma_k\left(I - C \epsilon_q R_q\right)
    \end{equation}
    where $\Gamma_k$ is as in Lemma \ref{lem:geom} and
    \begin{equation}\label{eq:C_value}
        \begin{split}
            C^{-1} : & = \int_{\R^2} 160\pi |\hat{\phi}(x)|^2 \left(\left(x_2 + \frac{1}{4}\right)^2 - \left(x_1 + \frac{1}{8}\right)^2\right)                                                       \\
                     & \times \left|\left(x_1k + x_2k^\perp\right) + \frac{1}{8}\left(k + 2k^\perp\right)\right| \left|\hat{K}_{\simeq 1}\left(5\left(x_1k + x_2k^\perp\right)\right)\right|^2\, dx.
        \end{split}
    \end{equation}
    And $\epsilon_q$ is chosen to satisfy
    \begin{equation*}
        \epsilon_q < \frac{\epsilon_{\Gamma}}{C\Vert R_q \Vert_{L^{\infty}}}
    \end{equation*}
    where $\epsilon_{\Gamma}$ is chosen to ensure that Lemma \ref{lem:geom} is satisfied for all $k \in \Omega$. Also for fixed $k \in \Omega$ put
    \begin{equation*}
        w_{q+1,k} = \nabla^{\perp}\left(\frac{1}{2\pi \sigma_{q+1}} \mathbb{P}_{\leq \lambda_{q+1}}\left(a_k(R_q(x)) \rho_{q+1}^{k}(x)\right) \left[e^{2\pi i \sigma_{q+1} (k + 2k^\perp) \cdot x} + e^{2\pi i \sigma_{q+1} (k - 2k^\perp) \cdot x}\right]\right)
    \end{equation*}
\end{definition}

\begin{remark}
    $C$ is nonzero due to $\phi$ being Schwartz and $\hat{K}_{\simeq 1}$ having compact support, and it is finite since $\hat{\phi}$ is analytic (recall $\phi$ has compact support), and so in particular it cannot vanish on the support of $\hat{K}_{\simeq 1}$. In addition, the support of $\hat{K}_{\simeq 1}$ forces every other term to be strictly positive.
\end{remark}

\begin{remark}
    Note that from Lemma \ref{lem:geom} and \ref{w:5} of Lemma \ref{lem:boldW} we have
    $$
        w_{q+1,k} + w_{q+1,-k} = \nabla^{\perp}\left(\frac{2}{\pi \sigma_{q+1}}\mathbb{P}_{\leq \lambda_{q+1}}\left(a_k(R_q(x)) \rho_{q+1}^{k}(x)\right) \cos\left(4\pi \sigma_{q+1}k^\perp \cdot x\right) \cos\left(2\pi \sigma_{q+1}k \cdot x\right)\right)
    $$
    and so in particular $w_{q+1}$ takes values in $\R^2$.
\end{remark}
\begin{remark}
    In view of Definition \ref{def:projs}, ~\eqref{eq:fourier_series}, and ~\eqref{eq:wq+1}, we see that we must have
    \begin{equation*}
        \lambda_{q+1} > 8^{\frac{1}{1-\epsilon}}
    \end{equation*}
    in order to ensure that $w_{q+1}$ contains highly oscillatory terms. Hence we take
    \begin{equation}\label{est:lambda_ineq}
        \lambda_{q+1} > 2^{\frac{100}{1-\epsilon}}.
    \end{equation}
\end{remark}
\begin{lemma}[\textbf{$L^p$ bounds for $w_{q+1}$}]
    Let $w_{q+1}$ be as in Definition \ref{def:wq+1}. Then for $1 \leq p \leq \infty$ we have
    \begin{equation}\label{eq:wq+1_est}
        \Vert w_{q+1} \Vert_{L^p} \lesssim \lambda_{q+1}^{-1/2} \lambda_{q+1}^{\left(1 - \epsilon\right)\left(1-\frac{2}{p}\right)}.
    \end{equation}
\end{lemma}
\begin{proof}
    Clearly $\Vert w_{q+1} \Vert_{L^p} \lesssim \sum_{k \in \Omega} \Vert w_{q+1,k} \Vert_{L^p}$ for every $k \in \Omega$, so it suffices to obtain the desired $L^p$ bounds for each $w_{q+1,k}$. So we write
    \begin{equation*}
        \begin{split}
            w_{q+1,k} & = \frac{1}{2\pi i \sigma_{q+1}} \mathbb{P}_{\leq \sigma_{q+1}}\left(\nabla^{\perp}\left(a_k(R_q) \rho_{q+1}^{k}\right)\right) \left(e^{2\pi i \sigma_{q+1} (k + 2k^\perp) \cdot x} + e^{2\pi i \sigma_{q+1} (k - 2k^\perp) \cdot x}\right) \\
                      & + \mathbb{P}_{\leq \lambda_{q+1}}\left(a_k(R_q) \rho_{q+1}^{k}\right) e^{2\pi i \sigma_{q+1} (k + 2k^\perp) \cdot x}(2k^\perp - 2k)                                                                                                        \\
                      & - \mathbb{P}_{\leq \lambda_{q+1}}\left(a_k(R_q) \rho_{q+1}^{k}\right) e^{2\pi i \sigma_{q+1} (k - 2k^\perp) \cdot x}(2k^\perp + 2k)                                                                                                        \\
                      & := w_{q+1,k}^{(1)} +  w_{q+1,k}^{(2)} + w_{q+1,k}^{(3)}.
        \end{split}
    \end{equation*}
    and then using Lemma \ref{lem:proj}, \ref{w:5} of Lemma \ref{lem:boldW}, and ~\eqref{eq:ak} we estimate
    \begin{equation}\label{eq:w^1_est}
        \begin{split}
            \Vert w_{q+1,k}^{(1)} \Vert_{L^p} & \lesssim \lambda_{q+1}^{-1} \left(\Vert \nabla^{\perp}a_k(R_q) \Vert_{L^{\infty}} \Vert \rho_{q+1}^{k} \Vert_{L^p} + \Vert a_k(R_q) \Vert_{L^{\infty}} \Vert \nabla^{\perp}\rho_{q+1}^{k} \Vert_{L^p}\right)            \\
                                              & \lesssim \lambda_{q+1}^{-1}\left(\lambda_{q+1}^{-1/2}\lambda_{q+1}^{\left(1-\epsilon\right)\left(1-\frac{2}{p}\right)} + \lambda_{q+1}^{-1/2}\lambda_{q+1}^{\left(1-\epsilon\right)\left(1-\frac{2}{p}\right)+1}\right) \\
                                              & \lesssim \lambda_{q+1}^{-1/2}\lambda_{q+1}^{\left(1-\epsilon\right)\left(1-\frac{2}{p}\right)}
        \end{split}
    \end{equation}
    and
    \begin{equation}\label{eq:w^2_est}
        \Vert w_{q+1,k}^{(2)} \Vert_{L^p} + \Vert w_{q+1,k}^{(3)} \Vert_{L^p} \lesssim \Vert a_k(R_q) \Vert_{L^{\infty}} \Vert \rho_{q+1}^k \Vert_{L^{p}} \lesssim \lambda_{q+1}^{-1/2}\lambda_{q+1}^{\left(1-\epsilon\right)\left(1-\frac{2}{p}\right)}.
    \end{equation}
    Combining ~\eqref{eq:w^1_est} and ~\eqref{eq:w^2_est} gives ~\eqref{eq:wq+1_est}.
\end{proof}

\subsection{Proof of Item \ref{i:2}}
Recall we set
\begin{equation}\label{eq:vq}
    v_{q+1} = v_q + w_{q+1}
\end{equation}
and
\begin{equation}\label{eq:uq}
    u_{q+1} = u_q + \Lambda w_{q+1}.
\end{equation}
We assume that $v_q$ and $u_q$ have mean-zero. From ~\eqref{eq:wq+1} we see that $w_{q+1}$ also has mean-zero. Clearly $\Lambda w_{q+1}$ has mean-zero, and since the sum of mean-zero functions still has mean-zero, we conclude that $v_{q+1}$ and $u_{q+1}$ have mean-zero. By the same argument, we conclude that $v_{q+1}$ and $u_{q+1}$ are divergence free.

\subsection{Proof of Item \ref{i:1}}
Assume that $(u_q,v_q,p_q,R_q)$ satisfy ~\eqref{eq:relax_Reynolds}, and we define the pressure $p_{q+1}$ by
\begin{equation}\label{eq:pq+1}
    \begin{split}
        p_{q+1} & = p_q - \frac{1}{2} \sum_{k \in \Omega}\left(\Lambda^{-1}(\nabla^{\perp} \cdot w_{q+1,k}) \nabla^{\perp} \cdot w_{q+1,-k} - \tilde{C}\lambda_{q+1} a^2_k(R_q)\right) \\
                & =: p_q + \tilde{p}_{q+1}.
    \end{split}
\end{equation}
where $\tilde{C}$ is given by ~\eqref{eq:const_2}
and $R_{q+1}$ to be a $2$ by $2$ symmetric matrix satisfying
\begin{equation}\label{eq:Rq+1}
    \begin{split}
        \operatorname{div}(R_{q+1}) = & \operatorname{div}(R_q) + u_q \cdot \nabla w_{q+1} + \Lambda w_{q+1} \cdot \nabla w_{q+1} + \Lambda w_{q+1} \cdot \nabla v_q \\
                                      & - (\nabla w_{q+1})^T \cdot u_q - (\nabla w_{q+1})^T \cdot \Lambda w_{q+1} - (\nabla v_q)^T \cdot \Lambda w_{q+1}             \\
                                      & + \Lambda^{\gamma} w_{q+1} + \nabla \tilde{p}_{q+1}.
    \end{split}
\end{equation}
For completeness we check that the right hand side of ~\eqref{eq:Rq+1} has zero mean. First notice that with the exception of the two terms $(\nabla w_{q+1})^T \cdot u_q + (\nabla v_q)^T \cdot \Lambda w_{q+1}$ and $(\nabla w_{q+1})^T \cdot \Lambda w_{q+1}$, all other terms trivially have zero mean since $v_q$, $u_q$, and $w_{q+1}$ are divergence free due to \ref{i:2} and then utilizing integration by parts. We now handle the two terms above as follows. For the first one using integration by parts and the fact that $\Lambda$ is self adjoint we have
\begin{equation*}
    \int_{\T^2} (\nabla w_{q+1})^T u_q = \int_{\T^2} \partial_j w^i_{q+1} u^i_q = - \int_{\T^2} w^i_{q+1} \partial_j u_q^i = -\int_{\T^2} \Lambda w^i_{q+1} \partial_j v^i_q = - \int_{\T^2} (\nabla v_q)^T \Lambda w_{q+1}
\end{equation*}
so
$$
    \int_{\T^2} (\nabla w_{q+1})^T u_q + \int_{\T^2} (\nabla v_q)^T \cdot \Lambda w_{q+1} = 0.
$$
For the second term we proceed similarly to get
\begin{equation*}
    \int_{\T^2} \partial_i w^j_{q+1}  \Lambda w^j_{q+1}
    = -\int_{\T^2} w^j_{q+1} \partial_i(\Lambda w^j_{q+1}) = -\int_{\T^2} w^j_{q+1} \Lambda(\partial_i w^j_{q+1}) = -\int_{\T^2} \partial_i w^j_{q+1} \Lambda w^j_{q+1}=0.
\end{equation*}
Since the right hand side of ~\eqref{eq:Rq+1} is mean-zero, such a matrix exists, and it is smooth. For our convenience, we define $R_{q+1}$ to be the matrix one obtains upon applying the inverse divergence $\mathcal{R}$ from Definition \ref{def:R} to the right hand side of ~\eqref{eq:Rq+1}. Then it is easy to check that if we define $(u_{q+1},v_{q+1},p_{q+1},R_{q+1})$ as in ~\eqref{eq:vq}, ~\eqref{eq:uq}, ~\eqref{eq:pq+1}, and ~\eqref{eq:Rq+1} then they satisfy ~\eqref{eq:relax_Reynolds} and are smooth.

\subsection{Proof of Item \ref{i:3}}\label{section_oscillation_error}
We rearrange ~\eqref{eq:Rq+1} to get
\begin{equation*}
    \begin{split}
        \operatorname{div}(R_{q+1}) = & \operatorname{div}(R_q) + \Lambda w_{q+1} \cdot \nabla w_{q+1} - ( \nabla w_{q+1})^T \cdot \Lambda w_{q+1} + \nabla \tilde{p}_{q+1}  \\
                                      & + \Lambda w_{q+1} \cdot \nabla v_q - (\nabla v_q)^T \cdot \Lambda w_{q+1} + u_q \cdot \nabla w_{q+1} - ( \nabla w_{q+1})^T \cdot u_q \\
                                      & + \Lambda^{\gamma} w_{q+1}.
    \end{split}
\end{equation*}
Let $\mathcal{R}$ be as in Definition \ref{def:R} and define
\begin{equation*}
    \begin{split}
        R_O & = R_q + \mathcal{R}(\Lambda w_{q+1} \cdot \nabla w_{q+1} - (\nabla w_{q+1})^T \cdot \Lambda w_{q+1} + \nabla \tilde{p}_{q+1}),                    \\
        R_N & = \mathcal{R}(\Lambda w_{q+1} \cdot \nabla v_q - (\nabla v_q)^T \cdot \Lambda w_{q+1} + u_q \cdot \nabla w_{q+1} - (\nabla w_{q+1})^T \cdot u_q), \\
        R_D & = \mathcal{R}(\Lambda^{\gamma} w_{q+1}).
    \end{split}
\end{equation*}
So then we have
$$
    R_{q+1} = R_O + R_N + R_D
$$
where $R_O$, $R_N$, and $R_D$ stand for the oscillation error, Nash error, and dissipation error respectively. We assume that $\Vert R_q \Vert_{\dot{H}^{-4}} < 2^{-q}$. We aim to show that $\Vert R_{q+1} \Vert_{\dot{H}^{-4}} < 2^{-q-1}$.

\noindent\texttt{Dissipation error: } We have
$$
    \Vert R_D \Vert_{\dot{H}^{-4}}^2 \lesssim \Vert \Lambda^{\gamma} w_{q+1} \Vert_{\dot{H}^{-5}}^2 = \sum_{j \not =0} |j|^{-10} |j|^{2\gamma} |\hat{w}_{q+1}(j)|^2 \leq \Vert w_{q+1} \Vert_{L^1}^2 \sum_{j \not =0} |j|^{2\gamma - 10}
$$
Since $\gamma \leq 2$ the sum is finite, so applying ~\eqref{est:lambda_ineq} and choosing $\lambda_{q+1}$ large enough we obtain
\begin{equation}\label{eq:dis_est}
    \Vert R_D \Vert_{\dot{H}^{-4}} \lesssim \lambda_{q+1}^{-1/2}\lambda_{q+1}^{(\epsilon-1)/2} < 2^{-2q-100}.
\end{equation}

\noindent\texttt{Nash error: } For $R_N$, we estimate each term separately. Now we choose $\lambda_{q+1}$ large enough such that
\begin{equation}\label{eq:spt_don_1}
    \operatorname{supp}\left(\left(\Lambda w_{q+1} \cdot \nabla v_q\right)^{\wedge}\right), \operatorname{supp}\left(\left((\nabla v_q)^T \cdot \Lambda w_{q+1}\right)^{\wedge}\right) \subset \left\{\xi : |\xi| > \frac{1}{2}\lambda_{q+1}\right\}
\end{equation}
and
\begin{equation}\label{eq:spt_cond_2}
    \operatorname{supp}\left(\left(u_q \cdot \nabla w_{q+1}\right)^{\wedge}\right), \operatorname{supp}\left(\left((\nabla w_{q+1})^T \cdot u_q\right)^{\wedge}\right) \subset \left\{\xi : |\xi| > \frac{1}{2}\lambda_{q+1}\right\}.
\end{equation}
So we have
\begin{equation}\label{eq:Nash_error_est1}
    \begin{split}
        \Vert \mathcal{R}(\Lambda w_{q+1} \cdot \nabla v_q) \Vert_{\dot{H}^{-4}}^2 & \lesssim \Vert \Lambda w_{q+1} \cdot \nabla v_q \Vert_{\dot{H}^{-5}}^2                                                 \\
                                                                                   & = \sum_{|j| \gtrsim \lambda_{q+1}} |j|^{-10} \left| \left(\Lambda w_{q+1} \cdot \nabla v_q\right)^{\wedge}(j)\right|^2 \\
                                                                                   & \lesssim \lambda_{q+1}^{-8} \left\Vert \Lambda w_{q+1}\right\Vert_{L^2}^2 \left\Vert\nabla v_q \right\Vert_{L^2}^2     \\
                                                                                   & \lesssim \lambda_{q+1}^{-8} \left(\lambda_{q+1}^{-1/2}\lambda_{q+1}\right)^2 \left(1\right)^2                          \\
                                                                                   & = \lambda_{q+1}^{-7}.
    \end{split}
\end{equation}
Notice in the above computation we bounded $\left\Vert\nabla v_q \right\Vert_{L^2}$ by $1$. This is due to the fact that $v_q$ is independent of the frequency $\lambda_{q+1}$. This is a fact that we will exploit multiple times throughout the rest of the paper. $\mathcal{R}((\nabla v_q)^T \cdot \Lambda w_{q+1})$, $\mathcal{R}(u_q \cdot \nabla w_{q+1})$, and $\mathcal{R}((\nabla w_{q+1})^T \cdot u_q)$ can all be handled in exactly the same manner. Hence from ~\eqref{eq:Nash_error_est1} (as well as the fact there are three more terms which obey identical bounds) we may choose $\lambda_{q+1}$ large enough such that
\begin{equation}\label{eq:Nash_error_est}
    \Vert R_N \Vert_{\dot{H}^{-4}} \lesssim \lambda_{q+1}^{-7/2} < 2^{-2q-100}.
\end{equation}

\noindent\texttt{Oscillation error: } This section is the most technical part of the paper, so we start by giving a brief overview. To obtain the desired estimate of $R_O$, we will follow the procedure introduced in \cite{BSV}, that is, we split $R_O$ into a high frequency component and a low frequency component. For the low frequency component, we will exploit the structure of the relaxed SQG momentum equation to decompose the divergence of this term as the sum of the divergence of a tensor and the gradient of a scalar valued function. The scalar valued function we regard as a pressure term and remove it using our definition of $\tilde{p}_{q+1}$. We then perform a further decomposition of the tensor product term into high and low frequency terms. From these low frequency terms, we are able to cancel the Reynolds stress $R_q$ and the remaining component of $\tilde{p}_{q+1}$. The remaining high frequency error terms from the tensor product term can be argued to be arbitrarily small in $\dot{H}^{-4}$ norm. For the high frequency component, utilizing a geometric argument we are able to deduce that these terms can be made negligible in the $\dot{H}^{-4}$ norm.

\noindent
Following \cite{BSV}, throughout this section we will shamelessly abuse notation and conflate the Fourier series and Fourier transform by identifying functions defined on $\T^2$ with their periodic extensions to $\R^2$. This will aid in the ease of the computations, but in practice there is no real harm; see \cite{CZ54} for a general transference principle.

\noindent We start with the low frequency term. This term corresponds precisely with the case when the frequency support of $\Lambda w_{q+1,k} \cdot \nabla w_{q+1,k'}$ and $(\nabla w_{q+1,k'})^T \cdot w_{q+1,k}$ is near the origin. Thus the sum over $k + k' = 0$ of such terms will form the low frequency component, and the sum over $k + k' \not = 0$ will form the high frequency component. Hence for $k \in \Omega$ set
\begin{equation}\label{eq:low_freq_term}
    \begin{split}
        \mathcal{T}_{q+1,k} = & \frac{1}{2}\bigg(\Lambda w_{q+1,k} \cdot \nabla w_{q+1,-k} - (\nabla w_{q+1,k})^T \cdot \Lambda w_{q+1,-k} \\
                              & + \Lambda w_{q+1,-k} \cdot \nabla w_{q+1,k} - (\nabla w_{q+1,-k})^T \cdot \Lambda w_{q+1,k}\bigg)
    \end{split}
\end{equation}
and recall from ~\eqref{eq:pq+1} that
\begin{equation*}
        \tilde{p}_{q+1} = - \frac{1}{2} \sum_{k \in \Omega} \Lambda^{-1}(\nabla^{\perp} \cdot w_{q+1,k}) \nabla^{\perp} \cdot w_{q+1,-k} - \frac{\tilde{C}\lambda_{q+1}}{2}\sum_{k \in \Omega} a^2_k(R_q)
\end{equation*}
where $\tilde{C}$ will be defined later in ~\eqref{eq:asym_const_2}. Now, denote the above by 
\begin{equation}\label{eq:pressure_decomp}
    \tilde{p}_{q+1} =: \sum_{k \in \Omega} p_{q+1,k,1} + \sum_{k \in \Omega} p_{q+1,k,2}.
\end{equation}
Our goal is to estimate the $\dot{H}^{-4}$ norm of $\mathcal{T}_{q+1,k} + p_{q+1,k,1} + p_{q+1,k,2}$. To obtain the desired decomposition of ~\eqref{eq:low_freq_term}, we follow closely the methodology of \cite[~pp.~1844-1851]{BSV}. Put $\vartheta_{q+1,k} = \nabla^{\perp} \cdot w_{q+1,k}$ so that using \cite[~Eq 5.20]{BSV} we have
$$
    \mathcal{T}_{q+1,k} = \frac{1}{2}\left( (R\vartheta_{q+1,k})\vartheta_{q+1,-k} + \vartheta_{q+1,k}(R\vartheta_{q+1,-k})\right).
$$
where $R = \Lambda^{-1} \nabla$ denotes the vector Riesz transform. Then, from \cite[~Equation 5.29]{BSV}, we have
\begin{equation}\label{eq:Tq+1k}
    \mathcal{T}_{q+1,k} = \frac{1}{2}\nabla\left(\Lambda^{-1} \vartheta_{q+1,k}\vartheta_{q+1,-k}\right) + \frac{1}{2}\operatorname{div}\left(S(\Lambda^{-1} \vartheta_{q+1,k}, R\vartheta_{q+1,-k})\right)
\end{equation}
where
\begin{equation}\label{eq:S^m}
    S^m(f,g) = \int_{\R^2} \int_{\R^2} s^m(\xi,\eta) \hat{f}(\xi) \hat{g}(\eta) e^{2\pi ix \cdot(\xi + \eta)}\, d\xi\, d\eta,
\end{equation}
and
\begin{equation}\label{eq:s^m}
    s^m(\xi,\eta) = \int_0^1 \frac{i((1-r) \eta - r\xi)^m}{|(1-r) \eta - r\xi|}\, dr.
\end{equation}
Clearly the first term in ~\eqref{eq:Tq+1k} is canceled by the first term in ~\eqref{eq:pressure_decomp}, so it suffices to bound
\begin{equation*}
    \frac{1}{2}S^m(\Lambda^{-1} \vartheta_{q+1,k}, R^{\ell}\vartheta_{q+1,-k}) - p_{q+1,k,2}
\end{equation*}
for $1 \leq m,\ell \leq 2$ in $\dot{H}^{-4}$ norm.

\noindent
For $k \in \Omega$ let
$$
    w_{q+1,k}^{\pm} = \nabla^\perp\left(\frac{1}{2\pi  \sigma_{q+1}} \mathbb{P}_{\leq \lambda_{q+1}}\left(a_k(R_q(x)) \rho_{q+1}^k(x)\right) e^{2\pi i \sigma_{q+1} (k \pm 2k^\perp) \cdot x} \right).
$$
Clearly $w_{q+1,k} = w_{q+1,k}^+ + w_{q+1,k}^-$. We also let
$$
    \vartheta_{q+1,k}^{\pm} = \nabla^\perp \cdot w_{q+1,k}^\pm.
$$
Since $S^m$ is a bilinear operator, we have the expansion
\begin{equation}\label{eq:S^m_decomp}
    \begin{split}
        S^m(\Lambda^{-1} \vartheta_{q+1,k}, R^{\ell}\vartheta_{q+1,-k}) & = S^m(\Lambda^{-1} \vartheta_{q+1,k}^+, R^{\ell}\vartheta_{q+1,-k}^+) + S^m(\Lambda^{-1} \vartheta_{q+1,k}^-, R^{\ell}\vartheta_{q+1,-k}^-) \\
                                                                        & + S^m(\Lambda^{-1} \vartheta_{q+1,k}^+, R^{\ell}\vartheta_{q+1,-k}^-) + S^m(\Lambda^{-1} \vartheta_{q+1,k}^-, R^{\ell}\vartheta_{q+1,-k}^+).
    \end{split}
\end{equation}
Heuristically, the top two terms on the right hand side of ~\eqref{eq:S^m_decomp} are of low frequency, and so we will utilize them to cancel the Reynolds stress. However, the two terms on the bottom of the right hand side of ~\eqref{eq:S^m_decomp} are of high frequency, and so we anticipate being able to make these terms arbitrarily small in $\dot{H}^{-4}$ norm.

\noindent We start by using the low frequency terms to cancel the Reynolds stress. Put
\begin{equation}\label{eq:Q^pm}
    \mathcal{Q}_{q+1,k}^{m,\ell,\pm} := \frac{1}{2}S^m(\Lambda^{-1} \vartheta^{\pm}_{q+1,k}, R^{\ell}\vartheta^{\pm}_{q+1,-k})
\end{equation}
and
\begin{equation*}
    \Phi_{q+1,k}^{\pm}(x)=\frac{1}{2\pi  \sigma_{q+1}} \mathbb{P}_{\leq \lambda_{q+1}}\left(a_k(R_q(x)) \rho_{q+1}^k(x)\right) e^{2\pi i \sigma_{q+1} (k \pm 2k^\perp) \cdot x}.
\end{equation*}
Now taking Fourier transforms using the multiplier of $\mathbb{P}_{\leq \lambda_{q+1}}$ and the frequency shift by $\sigma_{q+1} k$, we obtain

\begin{equation*}
    \widehat{\Phi}^{\pm}_{q+1,k}(\xi) = \frac{1}{2\pi  \sigma_{q+1}} \widehat{K}_{\simeq 1}\left(\frac{\xi - (k \pm 2k^\perp)\sigma_{q+1}}{\lambda_{q+1}}\right) (a_k(R_q) \rho_{q+1}^k)^{\wedge}(\xi - (k \pm 2k^\perp)\sigma_{q+1})
\end{equation*}
and replacing $k$ by $-k$ gives
\begin{equation*}
    \widehat{\Phi}^{\pm}_{q+1,-k}(\eta)= \frac{1}{2\pi  \sigma_{q+1}} \widehat{K}_{\simeq 1}\left(\frac{\eta  + (k \pm 2k^\perp)\sigma_{q+1}}{\lambda_{q+1}}\right) (a_k(R_q) \rho_{q+1}^k)^{\wedge}(\eta + (k\pm 2k^{\perp})\sigma_{q+1}).
\end{equation*}
Notice we have
\begin{equation*}
    \vartheta_{q+1,k}= \nabla^{\perp} \cdot \nabla^{\perp}\Phi_{q+1,k}(x) = \Delta \Phi_{q+1,k}(x)= -\Lambda^{2}\Phi_{q+1,k}(x).
\end{equation*}
Therefore, $\Lambda^{-1} \vartheta_{q+1,k} = -\Lambda \Phi_{q+1,k}$. Recalling $R^{\ell} = \Lambda^{-1} \partial_{\ell}$, by direct computation we get
\begin{equation}\label{eq:Rltheta}
    \begin{split}
        \widehat{R^{\ell} \vartheta^{\pm}_{q+1,-k}}(\eta) & = -\widehat{\Lambda \partial_{\ell} \Phi^{\pm}_{q+1,-k}}                                                                                                                                                             \\
                                                          & = - \frac{2\pi i |\eta| \eta^{\ell}}{\sigma_{q+1}} \hat{K}_{\simeq 1} \left(\frac{\eta +(k\pm2k^{\perp})\sigma_{q+1}}{\lambda_{q+1}} \right) (a_k(R_q) \rho^k_{q+1})^{\wedge} (\eta + (k\pm 2k^{\perp})\sigma_{q+1})
    \end{split}
\end{equation}
and
\begin{equation}\label{eq:Lambda^{-1}theta}
    \begin{split}
        \widehat{\Lambda^{-1} \vartheta^{\pm}_{q+1,k}}(\xi) & = -\widehat{\Lambda^{-1} \Phi^{\pm}_{q+1,k}}                                                                                                                                                \\
                                                            & = -\frac{|\xi|}{\sigma_{q+1}}\hat{K}_{\simeq 1} \left(\frac{\xi -(k\pm2k^{\perp})\sigma_{q+1}}{\lambda_{q+1}} \right) (a_k(R_q) \rho^k_{q+1})^{\wedge} (\xi -(k\pm 2k^{\perp})\sigma_{q+1}).
    \end{split}
\end{equation}
Then combining ~\eqref{eq:S^m}, ~\eqref{eq:s^m}, ~\eqref{eq:Q^pm}, ~\eqref{eq:Rltheta},  and ~\eqref{eq:Lambda^{-1}theta} we obtain

\begin{equation*}
    \begin{split}
        \mathcal{Q}_{q+1,k}^{m,\ell,\pm} & = \frac{1}{2} \int_{\R^2} \int_{\R^2} M_{q+1,k}^{m,\ell,\pm}(\xi,\eta) (a_k(R_q) \rho_{q+1}^{k})^{\wedge}(\eta)(a_k(R_q) \rho_{q+1}^{k})^{\wedge}(\xi) e^{2\pi ix \cdot (\eta + \xi)}\, d\xi\, d\eta
    \end{split}
\end{equation*}
where
\begin{equation}\label{eq:Mq+1kr}
    \begin{split}
        M^{m,\ell,\pm}_{q+1,k,r}(\xi,\eta) & = -\frac{2\pi}{\sigma^2_{q+1}} \frac{((1-r)\eta-r\xi-(k\pm2k^{\perp})\sigma_{q+1})^m}{|(1-r)\eta-r\xi-(k\pm2k^{\perp})\sigma_{q+1}|}\\
        &\times (\eta -(k\pm 2k^{\perp}\sigma_{q+1})^{\ell}
        |\xi + (k\pm 2k^{\perp})\sigma_{q+1}|\\
        &\times |\eta-(k\pm 2k^{\perp})\sigma_{q+1}|  \hat{K}_{\simeq1}\left(\frac{\eta}{\lambda_{q+1}}\right)\hat{K}_{\simeq1}\left(\frac{\xi}{\lambda_{q+1}}\right)
    \end{split}
\end{equation}
and
\begin{equation}\label{eq:Mq+1k}
    M_{q+1,k}^{m,\ell,\pm}(\xi,\eta) = \int_0^1 M_{q+1,k,r}^{m,\ell,\pm}(\xi,\eta)\, dr.
\end{equation}
Let us put
\begin{equation}\label{eq:M^*}
    \begin{split}
        M^{m,\ell,\pm*}_{k,r}(\xi,\eta) & = -2\pi \frac{((1-r)\eta - r\xi- (k \pm 2k^\perp))^m}{|(1-r)\eta - r\xi- (k \pm 2k^\perp)|}\\
        &\times(\eta^{\ell} - (k \pm 2k^\perp)^{\ell}) |\xi + (k \pm 2k^\perp)|\\
        &\times |\eta - (k \pm 2k^\perp)| \hat{K}_{\simeq 1}\left(\frac{5}{8}\xi\right) \hat{K}_{\simeq 1}\left(\frac{5}{8}\eta\right)
    \end{split}
\end{equation}
and note that
$$
    M_{q+1,k,r}^{m,\ell,\pm}(\xi,\eta) = \sigma_{q+1} M^{m,\ell,\pm*}_{k,r}\left(\frac{\xi}{\sigma_{q+1}},\frac{\eta}{\sigma_{q+1}}\right).
$$
Following the same reasoning as in \cite{BSV}, ~\eqref{eq:M^*} shows that $M^{m,\ell,\pm*}_{k,r}$ is independent of both $\sigma_{q+1}$ and $\lambda_{q+1}$, and is supported on $(\xi,\eta) \in B(0,1/5) \times B(0,1/5)$ in view of Definition \ref{def:projs}. Thus from geometric considerations we have
$$
    |\xi+(k \pm 2k^\perp)|,|\eta-(k \pm 2k^\perp)| \geq \sqrt{5} - \frac{1}{5} \geq 1
$$
\noindent
and

$$|(1-r)\eta - r\xi - (k \pm 2k^\perp)| \geq \sqrt{5} - \frac{2}{5} \geq 1.
$$

\noindent
Therefore, $M^{m,\ell,\pm*}_{k,r}$ is smooth, and can be bounded independently of $r \in (0,1)$.\\
Now, denote
\begin{equation}\label{eq:bilin_conv}
    \begin{split} \mathcal{Q}_{q+1,k}^{m,\ell,\pm} & = \frac{1}{2}  \int_0^1 \int_{\R^2} \int_{\R^2} K^{m,\ell,\pm}_{q+1,k,r}(x-y,x-z) (a_k(R_q) \rho_{q+1}^{k})(y)(a_k(R_q) \rho_{q+1}^{k})(z)\, dy\,dz\,dr
    \end{split}
\end{equation}
where
\begin{equation}\label{eq:kernel}
    K^{m,\ell,\pm}_{q+1,k,r}(y,z) = \sigma_{q+1}^5 \left( M^{m,\ell,\pm*}_{k,r}\right)^{\vee}\left(\sigma_{q+1}y,\sigma_{q+1}z\right).
\end{equation}
The explicit form of ~\eqref{eq:kernel}  will not play an important role for us; what is important is that
\begin{equation}\label{eq:kernel_est}
    \left\Vert y^{a} z^{b} K^{m,\ell,\pm}_{q+1,k,r} \right\Vert_{L^1_{y,z}} \lesssim \sigma_{q+1}^{1 - |a| - |b|}
\end{equation}
for multi-indices $a$ and $b$ with $|a| + |b| \in \{0,1,2\}$. This follows simply from change of variables in the integral and the fact $M^{m,\ell,\pm*}_{k,r}$ is Schwartz and independent of the $\sigma_{q+1}$ parameter. See \cite[Equation 5.46]{BSV}. From ~\eqref{eq:bilin_conv}, upon performing a change of variables and decomposing the product of the intermittent blobs as the sum of their mean with their mean free component we obtain
\begin{equation}\label{eq:bilin_conv_decomp}
    \begin{split}
        \mathcal{Q}_{q+1,k}^{m,\ell,\pm} = & \frac{1}{2}  \int_0^1 \int_{\R^2} \int_{\R^2} K^{m,\ell,\pm}_{q+1,k,r}(y,z) \mathbb{P}_{=0}\left(\rho_{q+1}^{k}(x-y) \rho_{q+1}^{k}(x-z)\right)\\
        &\times a_k(x-y)a_k(x-z)\, dy\,dz\,dr \\
        +& \frac{1}{2}  \int_0^1 \int_{\R^2} \int_{\R^2} K^{m,\ell,\pm}_{q+1,k,r}(y,z) \mathbb{P}_{\not=0}\left(\rho_{q+1}^{k}(x-y) \rho_{q+1}^{k}(x-z)\right)\\
        &\times a_k(x-y)a_k(x-z)\, dy\,dz\,dr.
    \end{split}
\end{equation}
Notice we have suppressed the dependence of the functions $a_k$ on the matrix $R_q$. Let us refer to expression on the right hand side of the top line of ~\eqref{eq:bilin_conv_decomp} as $\mathcal{Q}_{q+1,k}^{m,\ell,\pm,1}$ and the term on the second line as $\mathcal{Q}_{q+1,k}^{m,\ell,\pm,2}$. Using Lemma \ref{lem:fourier} and the fact $\phi$ is radially symmetric, we have that
\begin{equation*}
    \begin{split}
        \rho_{q+1}^k(x-y) \rho_{q+1}^k(x-z) = & \lambda_{q+1}^{2(\epsilon-1)} \sum_{n,m \in \Z^2} \hat{\phi}(\lambda_{q+1}^{\epsilon -1} n) \hat{\phi}(\lambda_{q+1}^{\epsilon -1} m)             \\
                                              & \times  e^{2\pi i 5 \lambda_{q+1}^{\epsilon}((n_1 +m_1)k +(n_2 + m_2)k^{\perp})\cdot x}                                                           \\
                                              & \times e^{-2\pi i 5 \lambda_{q+1}^{\epsilon}(n_1 k + n_2 k^{\perp})\cdot y} e^{-2\pi i 5 \lambda_{q+1}^{\epsilon}(m_1 k + m_2 k^{\perp}) \cdot z}.
    \end{split}
\end{equation*}
Hence
\begin{equation}\label{eq:off_mean_prod}
    \begin{split}
        \mathbb{P}_{\not=0,x}\left(\rho^{k}(x-y) \rho^{k}(x-z)\right) = & \lambda_{q+1}^{2(\epsilon -1)} \sum_{\substack{n+m \in \Z^2                                                                                       \\ n+m \not= 0}} \hat{\phi}(\lambda_{q+1}^{\epsilon -1} n) \hat{\phi}(\lambda_{q+1}^{\epsilon -1}m) \\
                                                                        & \times e^{-2\pi i 5 \lambda_{q+1}^{\epsilon}(n_1 k + n_2 k^{\perp})\cdot y} e^{-2\pi i 5 \lambda_{q+1}^{\epsilon}(m_1 k + m_2 k^{\perp}) \cdot z} \\
                                                                        & \times  e^{2\pi i 5 \lambda_{q+1}^{\epsilon}((n_1 +m_1)k +(n_2 + m_2)k^{\perp})\cdot x}
    \end{split}
\end{equation}
and
\begin{equation}\label{eq:mean_prod}
    \mathbb{P}_{=0,x}\left(\rho^{k}(x-y) \rho^{k}(x-z)\right) = \lambda_{q+1}^{2(\epsilon -1)} \sum_{n\in \Z^2} \left|\hat{\phi}(\lambda_{q+1}^{\epsilon-1} n) \right|^2 e^{-2\pi i 5 \lambda_{q+1}^{\epsilon}(n_1 k + n_2 k^{\perp})\cdot (y-z)}.
\end{equation}
Note we use the subscript $x$ in ~\eqref{eq:off_mean_prod} and ~\eqref{eq:mean_prod} to indicate the mean is taken with respect to the $x$ variable. Now utilizing Lemma \ref{lem:pseudo_diff_cont}, ~\eqref{eq:off_mean_prod}, and
\begin{equation*}
    \begin{split}
        \left\Vert a_k(x-y) a_k(x-z) \right\Vert_{C^4_x} \lesssim \left\Vert a_k(x-y) \right\Vert_{C^4_x} \left\Vert a_k(x-z) \right\Vert_{C^4_x}  = \left\Vert a_k \right\Vert_{C^4}^2  \lesssim \left\Vert a_k \right\Vert_{L^{\infty}}^2 \lesssim \lambda_{q+1}^{-1}
    \end{split}
\end{equation*}
gives
\begin{equation*}
    \begin{split}
         \bigg\Vert \mathbb{P}_{\not=0}\bigg(\rho_{q+1}^{k}(x-y) &\rho_{q+1}^{k}(x-z)\bigg)a_k(x-y)a_k(x-z)\bigg\Vert_{\dot{H}^{-4}_x} \\
         &\lesssim \lambda_{q+1}^{-1} \left\Vert \mathbb{P}_{\not=0,x}\left(\rho^{k}(x-y) \rho^{k}(x-z)\right) \right\Vert_{\dot{H}^{-4}_x}.
    \end{split}
\end{equation*}
Hence
\begin{equation}\label{eq:Q2_est}
\begin{split}
    \left\Vert \mathcal{Q}_{q+1,k}^{m,\ell,\pm,2} \right\Vert_{\dot{H}^{-4}} &\lesssim \lambda_{q+1}^{-1} \int_0^1 \int_{\R^2} \int_{\R^2} \left|K^{m,\ell,\pm}_{q+1,k,r}(y,z)\right|\\
    &\times \left\Vert \mathbb{P}_{\not=0,x}\left(\rho^{k}(x-y) \rho^{k}(x-z)\right) \right\Vert_{\dot{H}^{-4}_x}\, dy\, dz\, dr.
    \end{split}
\end{equation}
Setting
$$
    \mathcal{M}_{n,m}(y,z) = e^{-2\pi i 5 \lambda_{q+1}^{\epsilon}(n_1 k + n_2 k^{\perp})\cdot y} e^{-2\pi i 5 \lambda_{q+1}^{\epsilon}(m_1 k + m_2 k^{\perp}) \cdot z},
$$
we compute
\begin{equation*}
    \begin{split}
          & \left\Vert \mathbb{P}_{\not=0,x}\left(\rho^{k}(x-y) \rho^{k}(x-z)\right) \right\Vert_{\dot{H}^{-4}_x}^2                                                                                                                                     \\
        = & \sum_{j \in \Z^2 \setminus \{0\}} |j|^{-8} \left|\left(\mathbb{P}_{\not=0,x}\left(\rho^{k}(x-y) \rho^{k}(x-z)\right)\right)^{\wedge}(j)\right|^2                                                                                            \\
        = & \sum_{j \in \Z^2 \setminus \{0\}} |j|^{-8} \left| \sum_{(n,m) \in E_j} \lambda_{q+1}^{2(\epsilon-1)} \hat{\phi}\left(\lambda_{q+1}^{\epsilon-1} n\right) \hat{\phi}\left(\lambda_{q+1}^{\epsilon-1}m\right) \mathcal{M}_{n,m}(y,z)\right|^2
    \end{split}
\end{equation*}
where
$$
    E_j = \{(n,m) \in \Z^2 \times \Z^2 : 5\lambda_{q+1}^\epsilon\left((n_1+m_1)k + (n_2 + m_2)k^{\perp}\right) = j\}.
$$
For $k \in \Omega$, put $\mathcal{O}_k = \begin{bmatrix}
        k & k^{\perp}
    \end{bmatrix}$. $\mathcal{O}_k$ is orthogonal, and we have $(n,m) \in E_j$ if and only if $5\lambda_{q+1}^\epsilon(n+m) = \mathcal{O}_k^{-1}j = \mathcal{O}_k^Tj$. Hence $j \in 25\lambda_{q+1}^\epsilon \Z^2 \setminus \{0\}$ and so we have
\begin{equation*}
\begin{split}   
      \bigg\Vert \mathbb{P}_{\not=0,x}\bigg(&\rho^{k}(x-y) \rho^{k}(x-z)\bigg) \bigg\Vert_{\dot{H}^{-4}_x}^2\\
      =& \lambda_{q+1}^{4\epsilon-4}\sum_{j \in \Z^2 \setminus \{0\}} |25\lambda_{q+1}^{\epsilon} j|^{-8} \left| \sum_{n+m=5\mathcal{O}^{-1}_kj}  \hat{\phi}\left(\lambda_{q+1}^{\epsilon-1} n\right) \hat{\phi}\left(\lambda_{q+1}^{\epsilon-1}m\right) \mathcal{M}_{n,m}(y,z)\right|^2 
       \end{split}
\end{equation*}

\noindent
and we proceed to bound
    
\begin{equation}\label{eq:off_mean_comps}
    \begin{split}
         \bigg\Vert \mathbb{P}_{\not=0,x}\bigg(&\rho^{k}(x-y) \rho^{k}(x-z)\bigg) \bigg\Vert_{\dot{H}^{-4}_x}^2   \\
        \lesssim & \lambda_{q+1}^{-4\epsilon-4} \sum_{j \in \Z^2 \setminus \{0\}}  |j|^{-8} \left( \sum_{n+m=5\mathcal{O}_k^{-1}j} \left|\hat{\phi}\left(\lambda_{q+1}^{\epsilon-1}n\right)\right| \left|\hat{\phi}\left(\lambda_{q+1}^{\epsilon-1}m\right)\right|\right)^2                        \\
        \lesssim & \lambda_{q+1}^{-4\epsilon-4} \sum_{j \in \Z^2 \setminus \{0\}}  |j|^{-8} \left\Vert \left|\hat{\phi}\left(\lambda_{q+1}^{\epsilon-1} n\right)\right| \ast \left|\hat{\phi}\left(\lambda_{q+1}^{\epsilon-1} n\right)\right| \right\Vert_{\ell^{\infty}_n}^2\\
        \lesssim & \lambda_{q+1}^{-4\epsilon-4} \sum_{j \in \Z^2 \setminus \{0\}}  |j|^{-8} \left\Vert \hat{\phi}\left(\lambda_{q+1}^{\epsilon-1}n\right)\right\Vert_{\ell^2_n}^4.
        \end{split}
\end{equation}
Note to obtain the final line we utilize Young's convolution inequality. So now we have
\begin{equation}\label{eq:young}
    \begin{split}
        \left\Vert \hat{\phi}(\lambda_{q+1}^{\epsilon-1} \cdot) \right\Vert_{\ell^2}^4 & \lesssim \left(\sum_{n \in \Z^2} \left|\hat{\phi}(\lambda_{q+1}^{\epsilon-1}n)\right|^2 \right)^2                                                     \\
        & = \lambda_{q+1}^{2(1 - \epsilon)} \left(\sum_{n \in \Z^2}  \left|\hat{\phi}(\lambda_{q+1}^{\epsilon-1}n)\right|^2\lambda_{q+1}^{2(\epsilon-1)} \right)^2.
    \end{split}
\end{equation}
From the integral test and standard results on the convergence of improper Riemann integrals we obtain
\begin{equation}\label{eq:integral_comp}
    \sum_{n \in \Z^2}  \left|\hat{\phi}(\lambda_{q+1}^{\epsilon-1}n)\right|^2\lambda_{q+1}^{2(\epsilon-1)} \simeq \Vert \hat{\phi} \Vert_{L^2}^2 \simeq 1.
\end{equation}
Hence from ~\eqref{eq:off_mean_comps}, ~\eqref{eq:young}, and ~\eqref{eq:integral_comp} we deduce
\begin{equation}\label{eq:off_mean_est}
    \left\Vert \mathbb{P}_{\not=0,x}\left(\rho^{k}(x-y) \rho^{k}(x-z)\right) \right\Vert_{\dot{H}^{-4}_x} \lesssim \lambda_{q+1}^{-4\epsilon}
\end{equation}
and so from ~\eqref{eq:kernel_est}, ~\eqref{eq:Q2_est}, and ~\eqref{eq:off_mean_est} we obtain
\begin{equation}\label{eq:Q^2_est}
\begin{split}
    \left\Vert \mathcal{Q}_{q+1,k}^{m,\ell,\pm,2} \right\Vert_{\dot{H}^{-4}} &\lesssim \lambda_{q+1}^{-1} \int_0^1 \int_{\R^2} \int_{\R^2} \left|K^{m,\ell,\pm}_{q+1,k,r}(y,z)\right|\\
    &\times \left\Vert \mathbb{P}_{\not=0,x}\left(\rho^{k}(x-y) \rho^{k}(x-z)\right) \right\Vert_{\dot{H}^{-4}}\, dy\, dz\, dr\\
    &\lesssim \lambda_{q+1}^{-1-4\epsilon} \int_0^1 \int_{\R^2} \int_{\R^2} \left|K^{m,\ell,\pm}_{q+1,k,r}(y,z)\right|\, dy\, dz\, dr \\
    &\lesssim \lambda_{q+1}^{-4\epsilon}\\
    &< 2^{-2q-100}.
    \end{split}
\end{equation}
for $\lambda_{q+1}$ large enough.\\
Now, we utilize $\mathcal{Q}_{q+1,k}^{m,\ell,\pm,1}$ to cancel the Reynolds stress $R_q$. In order to achieve this, we write
\begin{equation*}
    a_k(x-y) = a_k(x) - y \cdot \int_0^1 \nabla a_k(x - ty)\, dt
\end{equation*}
and
\begin{equation*}
    a_k(x-z) = a_k(x) - z \cdot \int_0^1 \nabla a_k(x - tz)\, dt
\end{equation*}
to get
\begin{equation*}
    \begin{split}
        a_k(x-y)a_k(x-z) & = a^2_k(x) - a_k(x) y \cdot \int_0^1 \nabla a_k(x - ty)\, dt\\
        & - a_k(x)z \cdot \int_0^1 \nabla a_k(x - tz)\, dt\\
        & + \left(y \cdot \int_0^1 \nabla a_k(x - ty)\, dt\right)\\
        &\times\left(z \cdot \int_0^1 \nabla a_k(x - tz)\, dt\right).
    \end{split}
\end{equation*}
Thus
\begin{equation}\label{eq:Q^1_decomp}
    \begin{split}
        \mathcal{Q}_{q+1,k}^{m,\ell,\pm,1} & = \frac{1}{2}\int_0^1 \int_{\R^2} \int_{\R^2} K^{m,\ell,\pm}_{q+1,k,r}(y,z) \mathbb{P}_{=0,x}\left(\rho_{q+1}^{k}(x-y) \rho_{q+1}^{k}(x-z)\right)a^2_k(x)\, dy\,dz\,dr \\
        & - \frac{1}{2}\int_0^1 \int_0^1 \int_{\R^2} \int_{\R^2} K^{m,\ell,\pm}_{q+1,k,r}(y,z) \mathbb{P}_{=0,x}\left(\rho_{q+1}^{k}(x-y) \rho_{q+1}^{k}(x-z)\right)\\
        & \times a_k(x) y \cdot \nabla a_k(x - ty)\, dy\,dz\,dr\, dt\\
        & - \frac{1}{2}\int_0^1 \int_0^1 \int_{\R^2} \int_{\R^2} K^{m,\ell,\pm}_{q+1,k,r}(y,z) \mathbb{P}_{=0,x}\left(\rho_{q+1}^{k}(x-y) \rho_{q+1}^{k}(x-z)\right)\\
        & \times a_k(x) z \cdot \nabla a_k(x - tz)\, dy\,dz\,dr\, dt\\
        & + \frac{1}{2}\int_0^1 \int_{\R^2} \int_{\R^2} K^{m,\ell,\pm}_{q+1,k,r}(y,z) \mathbb{P}_{=0,x}\left(\rho_{q+1}^{k}(x-y) \rho_{q+1}^{k}(x-z)\right)\\
        & \times \prod_{w \in \{y,z\}}  \left(w \cdot \int_0^1 \nabla a_k(x - tw)\, dt\right) dy\,dz\,dr.
    \end{split}
\end{equation}
We denote the four terms on the right hand side of ~\eqref{eq:Q^1_decomp} by $\mathcal{Q}_{q+1,k}^{m,\ell,\pm,1,1}$, $\mathcal{Q}_{q+1,k}^{m,\ell,\pm,1,2}$, $\mathcal{Q}_{q+1,k}^{m,\ell,\pm,1,3}$, and $\mathcal{Q}_{q+1,k}^{m,\ell,\pm,1,4}$ respectively. Starting with $\mathcal{Q}_{q+1,k}^{m,\ell,\pm,1,1}$, using ~\eqref{eq:mean_prod} and properties of the Fourier transform we obtain
\begin{equation}\label{eq:Q^11_initial}
    \begin{split}
        \mathcal{Q}_{q+1,k}^{m,\ell,\pm,1,1} & = \frac{a^2_k(x)}{2} \sum_{n \in \Z^2} \lambda_{q+1}^{2(\epsilon-1)} \left|\hat{\phi}\left(\lambda_{q+1}^{\epsilon-1}n\right)\right|^2\\
        &\times \int_0^1 \int_{\R^2} \int_{\R^2} K^{m,\ell,\pm}_{q+1,k,r}(y,z) e^{-2\pi i 5 \lambda_{q+1}^{\epsilon}\mathcal{O}_kn\cdot (y-z)}\, dy\, dz\, dr \\
        & = \frac{a^2_k(x)}{2} \sum_{n \in \Z^2} \lambda_{q+1}^{2(\epsilon-1)} \left|\hat{\phi}\left(\lambda_{q+1}^{\epsilon-1}n\right)\right|^2\\
        &\times \int_0^1 \hat{K}^{m,\ell,\pm}_{q+1,k,r}\left(5\lambda_{q+1}^{\epsilon}\mathcal{O}_kn,-5\lambda_{q+1}^{\epsilon}\mathcal{O}_kn\right) dr.
    \end{split}
\end{equation}

\noindent Now from ~\eqref{eq:M^*} and ~\eqref{eq:kernel} we see that

\begin{equation*}
    \begin{split}       \hat{K}^{m,\ell,\pm}_{q+1,k,r}\left(\lambda_{q+1}^{\epsilon}5\mathcal{O}_kn,-\lambda_{q+1}^{\epsilon}5\mathcal{O}_kn\right) & = -\frac{2\pi}{\sigma_{q+1}^2} \left(5\lambda_{q+1}^\epsilon \mathcal{O}_kn + \sigma_{q+1}(k \pm 2k^\perp)\right)^m\\
    & \times \left(5\lambda_{q+1}^\epsilon\mathcal{O}_kn + \sigma_{q+1}(k \pm 2k^\perp)\right)^\ell\\
    & \times \left|5\lambda_{q+1}^\epsilon\mathcal{O}_kn + \sigma_{q+1}(k \pm 2k^\perp)\right|\\
    &\times \left|\hat{K}_{\simeq 1}\left(5\lambda_{q+1}^{\epsilon-1}\mathcal{O}_kn\right)\right|^2.
    \end{split}
\end{equation*}
Thus we have
\begin{equation}\label{eq:K^+}
    \begin{split}        & \hat{K}^{m,\ell,+}_{q+1,k,r}\left(\lambda_{q+1}^{\epsilon}5\mathcal{O}_kn,-\lambda_{q+1}^{\epsilon}5\mathcal{O}_kn\right)                                                                          \\
                     & = -\frac{2\pi}{\sigma_{q+1}^2} \bigg( \left(5\lambda_{q+1}^\epsilon n_1 + \sigma_{q+1}\right)^2 k^m k^\ell + \left(5\lambda_{q+1}^\epsilon n_2 + 2\sigma_{q+1}\right)^2 (k^\perp)^m (k^\perp)^\ell \\
                     & + \left(5\lambda_{q+1}^\epsilon n_1 + \sigma_{q+1}\right)\left(5\lambda_{q+1}^\epsilon n_2 + 2\sigma_{q+1}\right)\left(k^m (k^\perp)^\ell + (k^\perp)^m k^\ell\right)\bigg)                        \\
                     & \times \left|5\lambda_{q+1}^{\epsilon-1}\mathcal{O}_kn + \sigma_{q+1}(k + 2k^\perp)\right| \left|\hat{K}_{\simeq 1}\left(5\lambda_{q+1}^{\epsilon-1} \mathcal{O}_kn\right)\right|^2
    \end{split}
\end{equation}
and
\begin{equation}\label{eq:K^-1}
    \begin{split}
         & \hat{K}^{m,\ell,-}_{q+1,k,r}\left(\lambda_{q+1}^{\epsilon}5\mathcal{O}_kn,-\lambda_{q+1}^{\epsilon}5\mathcal{O}_kn\right)                                                                           \\
         & = -\frac{2\pi}{\sigma_{q+1}^2} \bigg( \left(5\lambda_{q+1}^\epsilon n_1 + \sigma_{q+1}\right)^2 k^m k^\ell + \left(5\lambda_{q+1}^\epsilon n_2 - 2\sigma_{q+1}\right)^2 (k^\perp)^m (k^\perp)^\ell \\
         & + \left(5\lambda_{q+1}^\epsilon n_1 + \sigma_{q+1}\right)\left(5\lambda_{q+1}^\epsilon n_2 - 2\sigma_{q+1}\right)\left(k^m (k^\perp)^\ell + (k^\perp)^m k^\ell\right)\bigg)                        \\
         & \times \left|5\lambda_{q+1}^{\epsilon-1}\mathcal{O}_kn + \sigma_{q+1}(k - 2k^\perp)\right| \left|\hat{K}_{\simeq 1}\left(5\lambda_{q+1}^{\epsilon-1} \mathcal{O}_kn\right)\right|^2.
    \end{split}
\end{equation}
If $n = (n_1,n_2)$, then setting $\tilde{n} = (n_1,-n_2)$ and applying ~\eqref{eq:K^-1} gives
\begin{equation}\label{eq:K^-2}
    \begin{split}
         & \hat{K}^{m,\ell,-}_{q+1,k,r}\left(\lambda_{q+1}^{\epsilon}5\mathcal{O}_k\tilde{n},-\lambda_{q+1}^{\epsilon}5\mathcal{O}_k\tilde{n}\right)                                                          \\
         & = -\frac{2\pi}{\sigma_{q+1}^2} \bigg( \left(5\lambda_{q+1}^\epsilon n_1 + \sigma_{q+1}\right)^2 k^m k^\ell + \left(5\lambda_{q+1}^\epsilon n_2 + 2\sigma_{q+1}\right)^2 (k^\perp)^m (k^\perp)^\ell \\
         & - \left(5\lambda_{q+1}^\epsilon n_1 + \sigma_{q+1}\right)\left(5\lambda_{q+1}^\epsilon n_2 + 2\sigma_{q+1}\right)\left(k^m (k^\perp)^\ell + (k^\perp)^m k^\ell\right)\bigg)                        \\
         & \times \left|5\lambda_{q+1}^{\epsilon-1}\mathcal{O}_kn + \sigma_{q+1}(k + 2k^\perp)\right| \left|\hat{K}_{\simeq 1}\left(5\lambda_{q+1}^{\epsilon-1} \mathcal{O}_kn\right)\right|^2.
    \end{split}
\end{equation}
So combining ~\eqref{eq:Q^11_initial}, ~\eqref{eq:K^+}, ~\eqref{eq:K^-1}, ~\eqref{eq:K^-2}, and the radial symmetry of $\phi$ gives
\begin{equation*}
    \begin{split}
         &\mathcal{Q}_{q+1,k}^{+,1} + \mathcal{Q}_{q+1,k}^{-,1} \\
         &= \frac{a^2_k(x)}{2} \sum_{n \in \Z^2} \lambda_{q+1}^{2(\epsilon-1)} \left|\hat{\phi}\left(\lambda_{q+1}^{\epsilon-1}n\right)\right|^2 \\
         &\times\int_0^1 \bigg(\hat{K}^{+}_{q+1,k,r}\left(\lambda_{q+1}^{\epsilon}5\mathcal{O}_kn,-\lambda_{q+1}^{\epsilon}5\mathcal{O}_kn\right) + \hat{K}^{-}_{q+1,k,r}\left(\lambda_{q+1}^{\epsilon}5\mathcal{O}_k\tilde{n},-\lambda_{q+1}^{\epsilon}5\mathcal{O}_k\tilde{n}\right)\bigg)\, dr\\
         & = -\frac{a^2_k(x)}{2} \sum_{n \in \Z^2} \frac{4\pi}{\sigma_{q+1}^2} \lambda_{q+1}^{2(\epsilon-1)} \left|\hat{\phi}\left(\lambda_{q+1}^{\epsilon-1}n\right)\right|^2 \bigg( \left(5\lambda_{q+1}^\epsilon n_1 + \sigma_{q+1}\right)^2 k \otimes k + \left(5\lambda_{q+1}^\epsilon n_2 + 2\sigma_{q+1}\right)^2 k^\perp \otimes k^\perp\bigg)\\
         &\times\left|5\lambda_{q+1}^{\epsilon-1}\mathcal{O}_kn + \sigma_{q+1}(k + 2k^\perp)\right| \left|\hat{K}_{\simeq 1}\left(5\lambda_{q+1}^{\epsilon-1} \mathcal{O}_kn\right)\right|^2
    \end{split}
\end{equation*}
where $\mathcal{Q}^{\pm,1}_{q+1,k}$ is the $2 \times 2$ matrix whose $m,\ell$ entry is $\mathcal{Q}^{m,\ell,\pm,1}_{q+1,k}$. Now we have
\begin{equation}\label{eq:Q^++Q^-2}
    \begin{split}
         \mathcal{Q}_{q+1,k}^{+,1} + \mathcal{Q}_{q+1,k}^{-,1}
         &= -\frac{\lambda_{q+1}a_k^2(x)}{2} k^\perp \otimes k^\perp \sum_{n \in \Z^2} 160\pi \left|\hat{\phi}\left(\lambda_{q+1}^{\epsilon-1}n\right)\right|^2 \bigg(\left(\lambda_{q+1}^{\epsilon-1} n_2 + \frac{1}{4}\right)^2\\
         &-\left(\lambda_{q+1}^{\epsilon-1} n_1 + \frac{1}{8}\right)^2 \bigg) \left|\lambda_{q+1}^{\epsilon-1}\mathcal{O}_kn + \frac{1}{8}(k + 2k^\perp)\right| \left|\hat{K}_{\simeq 1}\left(\lambda_{q+1}^{\epsilon-1} \mathcal{O}_kn\right)\right|^2 \lambda_{q+1}^{2(\epsilon-1)}\\
         & - \frac{\lambda_{q+1}a_k^2(x)}{2}I\sum_{n \in \Z^2} 160\pi \left|\hat{\phi}\left(\lambda_{q+1}^{\epsilon-1}n\right)\right|^2 \left(\lambda_{q+1}^{\epsilon-1} n_1 + \frac{1}{8}\right)^2\\
         & \times \left|\lambda_{q+1}^{\epsilon-1}\mathcal{O}_kn + \frac{1}{8}(k + 2k^\perp)\right| \left|\hat{K}_{\simeq 1}\left(5\lambda_{q+1}^{\epsilon-1} \mathcal{O}_kn\right)\right|^2 \lambda_{q+1}^{2(\epsilon-1)}
         \end{split}
\end{equation}
where we have used that $I = k \otimes k + k^{\perp} \otimes k^{\perp}$. Standard results in Riemann integration theory show that for $\lambda_{q+1}$ large we have that
\begin{equation}\label{eq:Riem_1}
    \begin{split}
         & \sum_{n \in \Z^2} 160\pi \left|\hat{\phi}\left(\lambda_{q+1}^{\epsilon-1}n\right)\right|^2 \left(\left(\lambda_{q+1}^{\epsilon-1} n_2 + \frac{1}{4}\right)^2 - \left(\lambda_{q+1}^{\epsilon-1} n_1 + \frac{1}{8}\right)^2 \right) \left|\lambda_{q+1}^{\epsilon-1}\mathcal{O}_kn + \frac{1}{8}(k + 2k^\perp)\right| \\
         & \times \left|\hat{K}_{\simeq 1}\left(5\lambda_{q+1}^{\epsilon-1} \mathcal{O}_kn\right)\right|^2 \lambda_{q+1}^{2(\epsilon-1)}                                                                                                                                                                                        \\
         & = \int_{\R^2} 160\pi |\hat{\phi}(x)|^2 \left(\left(x_2 + \frac{1}{4}\right)^2 - \left(x_1 + \frac{1}{8}\right)^2\right) \left|\mathcal{O}_kx + \frac{1}{8}\left(k + 2k^\perp\right)\right| \left|\hat{K}_{\simeq 1}\left(5\mathcal{O}_kx\right)\right|^2\, dx                                                        \\
         & + O\left(\lambda_{q+1}^{2(\epsilon-1)}\right)
    \end{split}
\end{equation}
and
\begin{equation}\label{eq:Riem_2}
    \begin{split}
         & \sum_{n \in \Z^2} 160\pi \left|\hat{\phi}\left(\lambda_{q+1}^{\epsilon-1}n\right)\right|^2 \left(\lambda_{q+1}^{\epsilon-1} n_1 + \frac{1}{8}\right)^2 \left|\lambda_{q+1}^{\epsilon-1}\mathcal{O}_kn + \frac{1}{8}(k + 2k^\perp)\right| \left|\hat{K}_{\simeq 1}\left(5\lambda_{q+1}^{\epsilon-1} \mathcal{O}_kn\right)\right|^2 \lambda_{q+1}^{2(\epsilon-1)} \\
         & = \int_{\R^2} 160\pi |\hat{\phi}(x)|^2 \left(x_1 + \frac{1}{8}\right)^2 \left|\mathcal{O}_kx + \frac{1}{8}\left(k + 2k^\perp\right)\right| \left|\hat{K}_{\simeq 1}\left(5\mathcal{O}_kx\right)\right|^2\, dx + O\left(\lambda_{q+1}^{2(\epsilon-1)}\right).
    \end{split}
\end{equation}
Recall from ~\eqref{eq:C_value} we had defined
\begin{equation*}
    \begin{split}
        C^{-1} &= \int_{\R^2} 160\pi |\hat{\phi}(x)|^2 \left(\left(x_2 + \frac{1}{4}\right)^2 - \left(x_1 + \frac{1}{8}\right)^2\right) \left|\mathcal{O}_kx + \frac{1}{8}\left(k + 2k^\perp\right)\right| \left|\hat{K}_{\simeq 1}\left(5\mathcal{O}_kx\right)\right|^2\, dx.
    \end{split}
\end{equation*}
Now set
\begin{equation}\label{eq:asym_const_1}
    \begin{split}
        C'_{q+1} & = \lambda_{q+1}^{2(1-\epsilon)}\bigg( \sum_{n \in \Z^2} 160\pi \left|\hat{\phi}\left(\lambda_{q+1}^{\epsilon-1}n\right)\right|^2 \left(\left(\lambda_{q+1}^{\epsilon-1} n_2 + \frac{1}{4}\right)^2 - \left(\lambda_{q+1}^{\epsilon-1} n_1 + \frac{1}{8}\right)^2 \right) \\
                 & \times \left|\lambda_{q+1}^{\epsilon-1}\mathcal{O}_kn + \frac{1}{8}(k + 2k^\perp)\right|
        \left|\hat{K}_{\simeq 1}\left(5\lambda_{q+1}^{\epsilon-1} \mathcal{O}_kn\right)\right|^2 \lambda_{q+1}^{2(\epsilon-1)} - C^{-1}\bigg),
    \end{split}
\end{equation}
\begin{equation}\label{eq:const_2}
    \tilde{C} = \int_{\R^2} 160\pi |\hat{\phi}(x)|^2 \left(x_1 + \frac{1}{8}\right)^2 \left|\mathcal{O}_kx + \frac{1}{8}\left(k + 2k^\perp\right)\right| \left|\hat{K}_{\simeq 1}\left(5\mathcal{O}_kx\right)\right|^2\, dx,
\end{equation}
and
\begin{equation}\label{eq:asym_const_2}
    \begin{split}
        \tilde{C}'_{q+1} & = \lambda_{q+1}^{2(1-\epsilon)}\bigg( \sum_{n \in \Z^2} 160\pi \left|\hat{\phi}\left(\lambda_{q+1}^{\epsilon-1}n\right)\right|^2 \left(\lambda_{q+1}^{\epsilon-1} n_1 + \frac{1}{8}\right)^2 \left|\lambda_{q+1}^{\epsilon-1}\mathcal{O}_kn + \frac{1}{8}(k + 2k^\perp)\right| \\
        & \times \left|\hat{K}_{\simeq 1}\left(5\lambda_{q+1}^{\epsilon-1} \mathcal{O}_kn\right)\right|^2 \lambda_{q+1}^{2(\epsilon-1)} - \tilde{C}\bigg).
    \end{split}
\end{equation}
Hence combining ~\eqref{eq:C_value}, ~\eqref{eq:Q^++Q^-2}, ~\eqref{eq:asym_const_1}, ~\eqref{eq:asym_const_2}, and ~\eqref{eq:const_2} we achieve
\begin{equation}\label{eq:Q^++Q^-2b}
    \begin{split}
        \sum_{k \in \Omega} \mathcal{Q}_{q+1,k}^{+,1} + \mathcal{Q}_{q+1,k}^{-,1} & = -\sum_{k \in \Omega} \frac{\lambda_{q+1}a_k^2(x)}{2} k^\perp \otimes k^\perp \left(C^{-1} + \lambda_{q+1}^{2(\epsilon-1)}C'_{q+1}\right) \\
                                                                                  & -\sum_{k \in \Omega} \frac{\lambda_{q+1}a_k^2(x)}{2}I \left(\tilde{C} + \lambda_{q+1}^{2(\epsilon-1)}\tilde{C}'_{q+1}\right).
    \end{split}
\end{equation}
We apply Lemma \ref{lem:geom} to the top line of ~\eqref{eq:Q^++Q^-2b} to obtain
\begin{equation}\label{eq:Rq_cancel}
    \begin{split}
        -\sum_{k \in \Omega} \frac{\lambda_{q+1}a_k^2(x)}{2} k^\perp \otimes k^\perp \left(C^{-1} + \lambda_{q+1}^{2(\epsilon-1)}C'_{q+1}\right) & = -\epsilon_q^{-1}\left(I + \epsilon_qCR_q\right)\left(C^{-1} + \lambda_{q+1}^{2(\epsilon-1)}C'_{q+1}\right) \\
                                                                                                                  & = -\epsilon_q^{-1}C^{-1}I - \epsilon_q\lambda_{q+1}^{2(\epsilon-1)}C_{q+1}'I                                 \\
                                                                                                                                                 & - R_q - \lambda_{q+1}^{2(\epsilon-1)}CC'_{q+1}R_q.
    \end{split}
\end{equation}
Note we have from ~\eqref{eq:Riem_1} and ~\eqref{eq:Riem_2} that $|C_{q+1}'| + |\tilde{C}'_{q+1}| \lesssim 1$. In ~\eqref{eq:pk2_est} we will show that
$$
    \lambda_{q+1} \Vert a_k^2 \Vert_{\dot{H}^{-4}} \lesssim \Vert R_q \Vert_{\dot{H}^{-4}}.
$$
Utilizing these two observations as well as ~\eqref{eq:Q^++Q^-2b}, ~\eqref{eq:Rq_cancel}, and the fact that the $\dot{H}^{-4}$ norm of constants is $0$ we see
\begin{equation}\label{eq:Q11_final_est}
    \begin{split}
        \left\Vert R_q + \sum_{k \in \Omega} \left(\mathcal{Q}_{q+1,k}^{+,1} + \mathcal{Q}_{q+1,k}^{-,1} + \frac{\lambda_{q+1}\tilde{C}}{2}a_k^2(x)I\right) \right\Vert_{\dot{H}^{-4}} & \lesssim \lambda_{q+1}^{2(\epsilon-1)}\left( |C| |C_{q+1}'| + \frac{|\tilde{C}'_{q+1}|}{2}\right) \Vert R_q\Vert_{\dot{H}^{-4}} \\
                                                                                                                                                                                       & \lesssim \lambda_{q+1}^{2(\epsilon-1)} 2^{-q}                                                                                   \\
                                                                                                                                                                                       & < 2^{-2q-100}
    \end{split}
\end{equation}
for $\lambda_{q+1}$ large enough.\\
Now, notice that $\mathcal{Q}_{q+1,k}^{m,\ell,\pm,1,2}$ and $\mathcal{Q}_{q+1,k}^{m,\ell,\pm,1,3}$ are symmetric, so it suffices to only bound one of them. We choose to focus our attention on $\mathcal{Q}_{q+1,k}^{m,\ell,\pm,1,2}$, which recall is given by
\begin{equation*}
    \begin{split}
         \mathcal{Q}_{q+1,k}^{m,\ell,\pm,1,2}&= - \frac{1}{2}\int_0^1 \int_0^1 \int_{\R^2} \int_{\R^2} K^{m,\ell,\pm}_{q+1,k,r}(y,z) \mathbb{P}_{=0,x}\left(\rho_{q+1}^{k}(x-y) \rho_{q+1}^{k}(x-z)\right) \\
        & \times a_k(x) y \cdot \nabla a_k(x - ty)\, dy\,dz\,dr\, dt.
    \end{split}
\end{equation*}
For this, using ~\eqref{eq:kernel_est} we have

\begin{equation*}
    \begin{split}
        \left\Vert \mathcal{Q}_{q+1,k}^{m,\ell,\pm,1,2} \right\Vert_{L^{\infty}}& \lesssim \int_0^1 \int_0^1 \int_{\R^2} \int_{\R^2} \left|yK^{m,\ell,\pm}_{q+1,k,r}(y,z)\right| \lambda_{q+1}^{-1}\, dy\, dz\, dr\, dt\\
        &\lesssim \lambda_{q+1}^{-1}.
    \end{split}
\end{equation*}
We have implicitly used that $\left|\mathbb{P}_{=0,x}\left(\rho_{q+1}^k(x-y) \rho_{q+1}^k(x-z)\right)\right| \lesssim 1$. Let us briefly justify this. Using ~\eqref{eq:mean_prod} and Riemann sum considerations we have that
\begin{equation}\label{eq:bdd_mean}
    \left|\mathbb{P}_{=0,x}\left(\rho_{q+1}^k(x-y) \rho_{q+1}^k(x-z)\right)\right| \lesssim \sum_{n \in \Z} \lambda_{q+1}^{\epsilon-1} \left|\hat{\phi}\left(\lambda_{q+1}^{\epsilon-1} n\right)\right|^2 \simeq \Vert \hat{\phi} \Vert_{L^2}^2 \simeq 1
\end{equation}
proving the claim. Hence, choosing $\lambda_{q+1}$ large enough we obtain that
\begin{equation}\label{eq:Q12Q13_est}
    \begin{split}
        \left\Vert \sum_{k \in \Omega} \left(\mathcal{Q}_{q+1,k}^{\pm,1,2} + \mathcal{Q}_{q+1,k}^{\pm,1,3}\right)\right\Vert_{\dot{H}^{-4}} &\lesssim \left\Vert \sum_{k \in \Omega} \left(\mathcal{Q}_{q+1,k}^{\pm,1,2} + \mathcal{Q}_{q+1,k}^{\pm,1,3}\right)\right\Vert_{L^{\infty}}\\
        &\lesssim \lambda_{q+1}^{-1}\\
        &< 2^{-2q-100}.
    \end{split}
\end{equation}
For the final term $\mathcal{Q}_{q+1,k}^{m,\ell,\pm,1,4}$, which recall is given by
\begin{equation*}
    \begin{split}
        \mathcal{Q}_{q+1,k}^{m,\ell,\pm,1,4} &= + \frac{1}{2}\int_0^1 \int_{\R^2} \int_{\R^2} K^{m,\ell,\pm}_{q+1,k,r}(y,z) \mathbb{P}_{=0,x}\left(\rho_{q+1}^{k}(x-y) \rho_{q+1}^{k}(x-z)\right) \\
        & \times \prod_{w \in \{y,z\}}  \left(w \cdot \int_0^1 \nabla a_k(x - tw)\, dt\right) dy\,dz\,dr.
    \end{split}
\end{equation*}
upon multiplying out the two dot products, a generic term will be of the form
\begin{equation*}
    \begin{split}
        \mathcal{Q}_{q+1,k}^{m,\ell,\pm,1,4,a,b} & :=\int_0^1 \int_0^1 \int_0^1 \int_{\R^2} \int_{\R^2} y^{a} z^{b} K^{m,\ell,\pm}_{q+1,k,r}(y,z) \mathbb{P}_{=0,x}\left(\rho_{q+1}^k(x-y) \rho_{q+1}^k(x-z)\right) \\
                                                 & \times \nabla^a a_k(x - t_1y)  \nabla^b a_k(x - t_2z)\, dy\,dz\,dr\, dt_1\, dt_2
    \end{split}
\end{equation*}
for $|a| + |b| = 2$. Hence applying the kernel estimate ~\eqref{eq:kernel_est} and ~\eqref{eq:bdd_mean} we achieve
\begin{equation*}
    \begin{split}
        \left\Vert \mathcal{Q}_{q+1,k}^{m,\ell,\pm,1,4,a,b} \right\Vert_{L^{\infty}} & \lesssim \int_0^1 \int_0^1 \int_0^1 \int_{\R^2} \int_{\R^2} \left|y^{a} z^{b} K^{m,\ell,\pm}_{q+1,k,r}(y,z)\right| \left|\mathbb{P}_{=0,x}\left(\rho_{q+1}^k(x-y) \rho_{q+1}^k(x-z)\right)\right| \\
                                                                                     & \times \Vert \nabla^{a} a_k(x - t_1y)  \nabla^{b} a_k(x - t_2z) \Vert_{L^{\infty}}\, dy\, dz\, dr\, dt_1\, dt_2                                                                                   \\
                                                                                     & \lesssim \int_0^1 \int_0^1 \int_0^1 \int_{\R^2} \int_{\R^2} \left|y^{a} z^{b} K^{m,\ell,\pm}_{q+1,k,r}(y,z)\right| \lambda_{q+1}^{-1}\, dy\, dz\, dr\, dt_1\, dt_2                                \\
                                                                                     & \lesssim \lambda_{q+1}^{-2}.
    \end{split}
\end{equation*}
Hence
\begin{equation}\label{eq:Q14_finalest}
\begin{split}
    \left\Vert \sum_{k \in \Omega} \mathcal{Q}_{q+1,k}^{\pm,1,4}\right\Vert_{\dot{H}^{-4}}
    \lesssim \sum_{k \in \Omega} \sum_{|a| + |b| = 2}  \left\Vert \mathcal{Q}_{q+1,k}^{\pm,1,4,a,b} \right\Vert_{L^{\infty}}\lesssim \lambda_{q+1}^{-2}< 2^{-2q-100}
    \end{split}
\end{equation}
for $\lambda_{q+1}$ large enough.
\noindent
Now we estimate the two remaining high frequency terms in ~\eqref{eq:S^m_decomp}. Put
\begin{equation*}\label{eq:Q^+-}
    \mathcal{Q}_{q+1,k}^{m,\ell,+,-} := \frac{1}{2}S^m(\Lambda^{-1} \vartheta^{+}_{q+1,k}, R^{\ell}\vartheta^{-}_{q+1,-k})
\end{equation*}
and
\begin{equation*}
    \mathcal{Q}_{q+1,k}^{m,\ell,-,+} := \frac{1}{2}S^m(\Lambda^{-1} \vartheta^{-}_{q+1,k}, R^{\ell}\vartheta^{+}_{q+1,-k}).
\end{equation*}
The analysis for $\mathcal{Q}_{q+1,k}^{m,\ell,-,+}$ is nearly identical to the procedure used to bound $\mathcal{Q}_{q+1,k}^{m,\ell,+,-}$, so we will only present the full details involved in bounding the latter term. Using entirely analogous procedures as to what is performed in ~\eqref{eq:Q^pm}-\eqref{eq:kernel_est}, one may obtain the following representation
\begin{equation*}
\begin{split}
    \mathcal{Q}_{q+1,k}^{m,\ell,+,-} &= \frac{e^{2\pi i 4k^\perp \sigma_{q+1} \cdot x}}{2} \int_0^1 \int_{\R^2} \int_{\R^2} K^{m,\ell,+,-}_{q+1,k,r}(x-y,x-z)\\
    &\times(a_k(R_q) \rho_{q+1}^{k})(y)(a_k(R_q) \rho_{q+1}^{k})(z)\, dy\,dz\,dr
    \end{split}
\end{equation*}
where
\begin{equation}\label{eq:M^+-*}
    \begin{split}
        M^{m,\ell,+,-*}_{k,r}(\xi,\eta) & = -2\pi \frac{((1-r)\eta - r\xi- (k + 2(4r-1)k^\perp))^m}{|(1-r)\eta - r\xi- (k + 2(4r-1)k^\perp)|}\\
        &\times (\eta^{\ell} - (k - 2k^\perp)^{\ell}) |\xi + (k + 2k^\perp)|\\
        &\times |\eta - (k - 2k^\perp)| \hat{K}_{\simeq 1}\left(\frac{5}{8}\xi\right) \hat{K}_{\simeq 1}\left(\frac{5}{8}\eta\right)
    \end{split}
\end{equation}
and
\begin{equation*}
    K^{m,\ell,+,-}_{q+1,k,r}(y,z) = \sigma_{q+1}^5 \left( M^{m,\ell,+,-*}_{k,r}\right)^{\vee}\left(\sigma_{q+1}y,\sigma_{q+1}z\right).
\end{equation*}
From ~\eqref{eq:M^+-*} and the fact
$$
|k + 2(4r-1)k^\perp| \geq 1 \quad \text{for all } r \in (0,1),
$$
we deduce that $M^{m,\ell,+,-*}_{k,r}$ is smooth, of compact support, and bounded independently of $r$. Hence $K^{m,\ell,+,-}_{q+1,k,r}$ obeys the same estimate as in ~\eqref{eq:kernel_est}. Importantly for us in this setting, this means that
\begin{equation}\label{eq:mult_est}
    \left\Vert \hat{K}^{m,\ell,+,-}_{q+1,k,r}(\xi,\eta) \right\Vert_{L^{\infty}} \lesssim \sigma_{q+1}
\end{equation}
and $\hat{K}^{m,\ell,+,-}_{q+1,k,r}$ is supported in $B(0,\lambda_{q+1}/5) \times B(0,\lambda_{q+1}/5)$. Now put
$$
    T(x) = \int_0^1 \int_{\R^2} \int_{\R^2} K^{m,\ell,+,-}_{q+1,k,r}(x-y,x-z) (a_k(R_q) \rho_{q+1}^{k})(y)(a_k(R_q) \rho_{q+1}^{k})(z)\, dy\,dz\,dr
$$
so that
$$
    2\mathcal{Q}_{q+1,k}^{m,\ell,+,-} = T(x)e^{2\pi i 4k^\perp \sigma_{q+1} \cdot x}.
$$
The Fourier transform of $T$ is given by
$$
 \widehat{T}(\zeta)  = \int_{\R^2} \int_{\R^2}\int_{\R^2} \widehat{K}^{m,\ell,+,-}_{q+1,k,r}(\xi,\eta) (a_k(R_q)\rho^k)^{\wedge}(\xi)(a_k(R_q)\rho^k)^{\wedge}(\eta) e^{2\pi i x \cdot (\xi + \eta -\zeta)} dx \, d\xi \, d\eta.
$$
In the distributional sense, we have that
$$
\int_{\R^2} e^{2\pi i x \cdot (\xi + \eta -\zeta)}\, dx = \delta(\xi + \eta - \zeta).
$$
Thus
\begin{equation*}
    \begin{split}
        \widehat{T}(\zeta) & = \int_{\R^2} \int_{\R^2} \delta(\xi+\eta-\zeta) \widehat{K}^{m,\ell,+,-}_{q+1,k,r}(\xi,\eta)(a_k(R_q)\rho^k)^{\wedge}(\xi)(a_k(R_q)\rho^k)^{\wedge}(\eta) d\xi \, d\eta\\
        & = \int_{\R^2} \widehat{K}^{m,\ell,+,-}_{q+1,k,r}(\xi , \zeta - \xi)(a_k(R_q)\rho^k)^{\wedge}(\xi) (a_k(R_q)\rho^k)^{\wedge}(\zeta - \xi) d\xi.
    \end{split}
\end{equation*}
The support of $\hat{K}^{m,\ell,+,-}_{q+1,k,r}$ forces
$$
    |\xi| \leq \frac{\lambda_{q+1}}{5} \quad \text{and} \quad |\zeta - \xi| \leq \frac{\lambda_{q+1}}{5}.
$$
Thus
$$
    |\zeta| \leq |\zeta - \xi| + |\xi| \leq \frac{2\lambda_{q+1}}{5}.
$$
So, $\hat{T}$ is supported in $B(0,2\lambda_{q+1}/5)$. Using this as well as ~\eqref{eq:mult_est}, we compute
\begin{equation*}
    \begin{split}
        \Vert \mathcal{Q}_{q+1,k}^{m,\ell,+,-} \Vert_{\dot{H}^{-4}}^2 & \lesssim \Vert T(x)e^{2\pi i 4k^\perp \sigma_{q+1} \cdot x} \Vert_{\dot{H}^{-4}}^2                                  \\
                                                                      & = \sum_{\zeta \in \Z^2 \setminus \{0\}} |\zeta|^{-8} |\hat{T}(\zeta - 4\sigma_{q+1} k^\perp)|^2                     \\
                                                                      & = \sum_{\zeta \in B(4\sigma_{q+1}k^\perp,2\lambda_{q+1}/5)} |\zeta|^{-8} |\hat{T}(\zeta - 4\sigma_{q+1} k^\perp)|^2 \\
                                                                      & \simeq \lambda_{q+1}^{-8} \sum_{\zeta \in B(0,2\lambda_{q+1}/5)} |\hat{T}(\zeta)|^2                                 \\
                                                                      & \lesssim \lambda_{q+1}^{-6} \sum_{\zeta \in B(0,2\lambda_{q+1}/5)} \left\Vert a_k(R_q) \rho^k\right\Vert_{L^2}^4    \\
                                                                      & \lesssim \lambda_{q+1}^{-6} \lambda_{q+1}^2 \left(\lambda_{q+1}^{-1/2}\right)^4                                     \\
                                                                      & = \lambda_{q+1}^{-6}.
    \end{split}
\end{equation*}
So using this estimate as well as applying an identical procedure for $\mathcal{Q}_{q+1,k}^{m,\ell,-,+}$, we obtain, for $\lambda_{q+1}$ large enough, that
\begin{equation}\label{eq:high_freq_ests}
    \left\Vert \mathcal{Q}_{q+1,k}^{m,\ell,+,-} \right\Vert_{\dot{H}^{-4}} +  \left\Vert \mathcal{Q}_{q+1,k}^{m,\ell,-,+} \right\Vert_{\dot{H}^{-4}} \lesssim \lambda_{q+1}^{-3} < 2^{-2q-100}.
\end{equation}
Hence combining ~\eqref{eq:Q^2_est}, ~\eqref{eq:Q^1_decomp}, ~\eqref{eq:Q11_final_est}, ~\eqref{eq:Q12Q13_est}, ~\eqref{eq:Q14_finalest}, and ~\eqref{eq:high_freq_ests} we obtain
\begin{equation}\label{eq:Q_final_est}
    \begin{split}
        \bigg\Vert R_q +  \sum_{k \in \Omega} \bigg(\mathcal{Q}^+_{q+1,k} + \mathcal{Q}^-_{q+1,k} + \mathcal{Q}^{+,+}_{q+1,k} &+ \mathcal{Q}^{-,-}_{q+1,k}\bigg) + \tilde{p}_{q+1}I \bigg\Vert_{\dot{H}^{-4}}\\
        &< (32)2^{-2q-100} = 2^{-2q-95}
    \end{split}
\end{equation}
which completes the proof of the boundedness of the low frequency term.\\
Now we turn our attention towards the high frequency component. So we estimate
\begin{equation*}
\begin{split}
    &\left\Vert \sum_{k+k' \not = 0}  \mathcal{R}\left(\Lambda w_{q+1,k} \cdot \nabla w_{q+1,k'} - (\nabla w_{q+1,k})^T \cdot \Lambda w_{q+1,k'}\right)\right\Vert_{\dot{H}^{-4}}\\
    \lesssim& \sum_{k+k' \not = 0} \left\Vert\left(R\vartheta_{q+1,k}\right)\vartheta_{q+1,k'} \right\Vert_{\dot{H}^{-4}}.
    \end{split}
\end{equation*}
Fix $k,k' \in \Omega$ such that $k \not = - k'$. Then we have that
\begin{equation}\label{eq:Rtheta_theta_est}
    \begin{split}
        \left\Vert (R\vartheta_{q+1,k})\vartheta_{q+1,k'} \right\Vert_{\dot{H}^{-4}}^2 & = \sum_{j \not = 0} |j|^{-8} \left|\sum_{n \in \Z^2} (R\vartheta_{q+1,k})^{\wedge}(n) \hat{\vartheta}_{q+1,k'}(j-n)\right|^2         \\
                                                                                       & = \sum_{j \not = 0} |j|^{-8} \left|\sum_{n \in \Z^2} \frac{in}{|n|}\hat{\vartheta}_{q+1,k}(n) \hat{\vartheta}_{q+1,k'}(j-n)\right|^2.
    \end{split}
\end{equation}
Note that $\vartheta_{q+1,k}$ has frequency support in the ball $B((k + 2k^\perp)\sigma_{q+1},\lambda_{q+1}/8)$. Hence the inner sum of ~\eqref{eq:Rtheta_theta_est} is nonzero only when we have
$$
    |n - \sigma_{q+1}(k+2k^\perp)| \leq \frac{1}{8}\lambda_{q+1}
$$
and
$$
|j - n - \sigma_{q+1}(k' + 2(k')^\perp)| \leq \frac{1}{8}\lambda_{q+1}.
$$
From the triangle inequality we deduce that $j \in B(\sigma_{q+1}(k + 2k^\perp+k' + 2(k')^\perp),\lambda_{q+1}/4)$. Since $k + k' \not = 0$, from Lemma \ref{lem:geom} we must have that $|k + k'| \geq \frac{1}{2}$. Thus
\begin{equation*}
    |j| \geq \left(\frac{5\sqrt{5}}{16} - \frac{1}{4}\right)\lambda_{q+1} \simeq \lambda_{q+1}.
\end{equation*}
We also have that
$$
    |\Z^2 \setminus \{0\} \cap B(\sigma_{q+1}(k + 2k^\perp + k' + 2(k')^\perp), \lambda_{q+1}/4)| \simeq \lambda_{q+1}^2.
$$
Now using this as well as the $L^2$ boundedness of the Riesz transform and
$$
    \Vert \vartheta_{q+1,k} \Vert_{L^2} \lesssim \lambda_{q+1}
$$
we have
\begin{equation}\label{eq:final_est}
    \begin{split}
        \left\Vert (R\vartheta_{q+1,k})\vartheta_{q+1,k'} \right\Vert_{\dot{H}^{-4}}^2 & = \sum_{j \not = 0} |j|^{-8} |((R\vartheta_{q+1,k})\vartheta_{q+1,k'})^{\wedge}(j)|^2                   \\
                                                                                       & \lesssim \Vert (R\vartheta_{q+1,k})\vartheta_{q+1,k'} \Vert_{L^1}^2 \lambda_{q+1}^2 \lambda_{q+1}^{-8}  \\
                                                                                       & \leq \Vert (R\vartheta_{q+1,k}) \Vert_{L^2}^2 \Vert \vartheta_{q+1,k'} \Vert_{L^2}^2 \lambda_{q+1}^{-8} \\
                                                                                       & \lesssim (\lambda_{q+1}^2\lambda_{q+1}^{-1})^2 \lambda_{q+1}^{-8}                                       \\
                                                                                       & = \lambda_{q+1}^{-6}
    \end{split}
\end{equation}
which goes to $0$ as $\lambda_{q+1} \to \infty$. Hence from ~\eqref{eq:final_est} we may choose $\lambda_{q+1}$ large enough such that
\begin{equation}\label{eq:leftover_terms_est}
    \left\Vert \sum_{k + k' \not = 0} \mathcal{R}\left((R\vartheta_{q+1,k})\vartheta_{q+1,k'}\right) \right\Vert_{\dot{H}^{-4}} < 2^{-2q-100}.
\end{equation}
We see ~\eqref{eq:Q_final_est} and ~\eqref{eq:leftover_terms_est} combine to give
\begin{equation}\label{eq:RO_est}
\begin{split}
    \Vert R_O \Vert_{\dot{H}^{-4}} &\lesssim \left\Vert R_q + \sum_{k \in \Omega} \mathcal{R}(\mathcal{T}_{q+1,k}) + \tilde{p}_{q+1}I \right\Vert_{\dot{H}^{-4}}\\
    &+ \left\Vert \sum_{k + k' \not = 0} \mathcal{R}\left((R\vartheta_{q+1,k})\vartheta_{q+1,k'}\right) \right\Vert_{\dot{H}^{-4}}\\
    &< 2^{-2q-80}.
    \end{split}
\end{equation}
Finally, from ~\eqref{eq:dis_est}, ~\eqref{eq:Nash_error_est} and ~\eqref{eq:RO_est}, we have
\begin{equation}\label{eq:Reynolds_est_final}
    \Vert R_{q+1} \Vert_{\dot{H}^{-4}} < (2)2^{-2q-100} + 2^{-2q-80} < 2^{-2q-40} < 2^{-q-1}.
\end{equation}
From ~\eqref{eq:Reynolds_est_final} we deduce \ref{i:3}.

\subsection{Proof of Item \ref{i:4}}
We utilize the following variant of the classical Bernstein inequality.

\begin{lemma}\label{lem:Bernstein}
    Fix $\alpha \in \R$ and $\lambda > 0$ large. Suppose $u:\T^2 \to \R$ is smooth and
    $\operatorname{supp}(\hat{u}) \subset \{\xi : |\xi| \simeq \lambda\}$. Then
    \begin{equation}\label{eq:C^alpha_est}
        \Vert u \Vert_{\dot{B}^\alpha_{\infty,\infty}} \lesssim \lambda^\alpha \Vert u \Vert_{L^{\infty}}
    \end{equation}
    and
    \begin{equation}\label{eq:H^alpha_est}
        \Vert u \Vert_{\dot{H}^\alpha} \lesssim \lambda^\alpha \Vert u \Vert_{L^2}.
    \end{equation}
\end{lemma}
\begin{proof}
    Recall that
    \begin{equation*}
        \Vert u \Vert_{\dot{B}^\alpha_{\infty,\infty}} = \sup_{j \geq 0} 2^{\alpha j} \Vert \mathbb{P}_{2^j}(u) \Vert_{L^{\infty}}.
    \end{equation*}
    Since the frequency support of $u$ is contained in an annulus of radius $\lambda$, then $\mathbb{P}_{2^j}(u) = 0$ unless $\lambda \simeq 2^j$. Thus
    $$
        \Vert u \Vert_{\dot{B}^\alpha_{\infty,\infty}} \simeq \lambda^\alpha \Vert \mathbb{P}_{\lambda}(u) \Vert_{L^{\infty}} \leq \lambda^\alpha \Vert u \Vert_{L^{\infty}},
    $$
    giving ~\eqref{eq:C^alpha_est}.Note that ~\eqref{eq:H^alpha_est} is obtained similarly.
\end{proof}

\noindent
From \ref{i:6} we have that frequency support of $w_{q+1}$ is contained within annulus $\{\xi : |\xi| \simeq \lambda_{q+1}\}$ and from \ref{i:1} it is smooth, and thus Lemma ~\ref{lem:Bernstein} applies. Now applying ~\eqref{eq:wq+1_est} and ~\eqref{eq:C^alpha_est} we obtain
\begin{equation}\label{eq:C^alpha_est_2}
    \Vert v_{q+1} - v_q \Vert_{\dot{B}^\alpha_{\infty,\infty}} = \Vert w_{q+1} \Vert_{\dot{B}^\alpha_{\infty,\infty}} \lesssim \lambda_{q+1}^{\alpha} \lambda_{q+1}^{-1/2} \lambda_{q+1}^{1-\epsilon}.
\end{equation}
From ~\eqref{eq:beta_epsilon_est}, the exponent in ~\eqref{eq:C^alpha_est_2} is negative, and thus we may choose $\lambda_{q+1}$ large enough to ensure that
\begin{equation}\label{eq:C^alpha_est_3}
    \Vert v_{q+1} - v_q \Vert_{\dot{B}^\alpha_{\infty,\infty}} < 2^{-q-10}.
\end{equation}
Similarly applying ~\eqref{eq:H^alpha_est} we get
\begin{equation}\label{eq:H^alpha_est_2}
    \Vert v_{q+1} - v_q \Vert_{\dot{H}^\alpha} = \Vert w_{q+1} \Vert_{\dot{H}^\alpha} \lesssim \lambda_{q+1}^{\alpha} \lambda_{q+1}^{-1/2}.
\end{equation}
Again applying ~\eqref{eq:beta_epsilon_est}, the exponent in ~\eqref{eq:H^alpha_est_2} is negative, and thus we may choose $\lambda_{q+1}$ large enough to get
\begin{equation}\label{eq:H^alpha_est_3}
    \Vert v_{q+1} - v_q \Vert_{\dot{H}^\alpha} < 2^{-q-10}.
\end{equation}
Combining ~\eqref{eq:C^alpha_est_3} and ~\eqref{eq:H^alpha_est_3} gives \ref{i:4} at level $q+1$.
\subsection{Proof of Item \ref{i:5}}
Recall we assume that $\Vert v_{q} \Vert_{L^1} > (1+2^{-q})\delta$, and we want to prove that $\Vert v_{q+1} \Vert_{L^1} > (1 + 2^{-q-1})\delta$. So we have
\begin{equation}\label{eq:L1_vq_lowerest}
    \Vert v_{q+1} \Vert_{L^1} \geq \Vert v_q \Vert_{L^1} - \Vert w_{q+1} \Vert_{L^1} > (1 + 2^{-q-1})\delta - \Vert w_{q+1} \Vert_{L^1}.
\end{equation}
From ~\eqref{eq:wq+1_est} we may choose $\lambda_{q+1}$ large enough such that
\begin{equation}\label{eq:wq+1_est_delta}
    \Vert w_{q+1} \Vert_{L^1} < 2^{-q-2}\delta.
\end{equation}
With this choice of $\lambda_{q+1}$, combining ~\eqref{eq:L1_vq_lowerest} and ~\eqref{eq:wq+1_est_delta} gives $\Vert v_{q+1} \Vert_{L^1} \geq (1 + 2^{-q-1})\delta$.

\subsection{Proof of Item \ref{i:6}}\label{Sec:freq_support} From Definition \ref{def:wq+1}, it is clear that
\begin{equation*}
    \begin{split}
        \operatorname{supp}(\hat{w}_{q+1}) & = \bigcup_{k \in \Omega} B(\sigma_{q+1}(k + 2k^\perp),\lambda_{q+1}/8)\\
        & \subset \left\{\xi \in \Z^2 : \left(\frac{5\sqrt{5}}{8}-\frac{1}{4}\right)\lambda_{q+1} \leq |\xi| \leq \left(\frac{5\sqrt{5}}{8}+\frac{1}{4}\right)\lambda_{q+1}\right\}.
    \end{split}
\end{equation*}
Note the final set containment above comes from our choice $\sigma_{q+1} = \frac{5}{8}\lambda_{q+1}$. Since
$$
    1 < \frac{5\sqrt{5}}{8}-\frac{1}{4} < \frac{5\sqrt{5}}{8}+\frac{1}{4} < \frac{12}{7}
$$
from Definition \ref{def:projs}, we see that the only value of $j$ such that $\mathbb{P}_{2^j}(w_{q+1}) \not =0$ is when\\ $j = \log_2(\lambda_{q+1})$ and from the frequency support, $\mathbb{P}_{\lambda_{q+1}}(w_{q+1}) = w_{q+1}$.

\subsection{Proof of Item \ref{i:7}}\label{Sec:weak_para}
Clearly
\begin{equation*}
    \begin{split}
        \sum_{\substack{n,m \leq q+1                                                            \\n\not=m}} \Vert \Lambda w_m \nabla^{\perp}\cdot w_n \Vert_{\dot{H}^{-5}} &= \sum_{\substack{n,m \leq q\\n \not = m}} \Vert \Lambda w_m \nabla^{\perp}\cdot w_n \Vert_{\dot{H}^{-5}}\\
         & + \sum_{n \leq q} \Vert \Lambda w_{q+1} \nabla^{\perp}\cdot w_n \Vert_{\dot{H}^{-5}} \\
         & + \sum_{m \leq q} \Vert \Lambda w_m \nabla^{\perp}\cdot w_{q+1} \Vert_{\dot{H}^{-5}}.
    \end{split}
\end{equation*}
First, we deal with the off-diagonal terms. Recall, from our inductive assumption we have
$$
    \sum_{\substack{n,m \leq q\\n \not = m}} \Vert \Lambda w_m \nabla^{\perp}\cdot w_n \Vert_{\dot{H}^{-5}} < C_1 - 2^{-q}.
$$
So then
\begin{equation*}
    \begin{split}
        \Vert \Lambda w_{q+1} \nabla^{\perp}\cdot w_n \Vert_{\dot{H}^{-5}}^2 & \simeq \sum_{j \not =0} |j|^{-10} \left|\sum_{j' \in \Z^2} |j'| \hat{w}_{q+1}(j') (j - j')^{\perp} \hat{w}_n(j-j')\right|^2 \\
                                                                             & \lesssim \sum_{j \not =0} |j|^{-10} \left(\sum_{j' \in \Z^2} |j'||\hat{w}_{q+1}(j')| |j-j'|\hat{w}_n(j-j')|\right)^2.
    \end{split}
\end{equation*}
Notice from the above, in order for the sum to be nonzero, we must have that $|j'| \simeq \lambda_{q+1}$ and $|j-j'| \simeq \lambda_n$. But since $\lambda_{q+1} \gg \lambda_n$, this forces $|j| \simeq \lambda_{q+1}$. Thus utilizing that $|\hat{w}_{q+1}| \leq \Vert w \Vert_{L^2} \lesssim 1$, we have
\begin{equation*}
    \begin{split}
        \Vert \Lambda w_{q+1} \nabla^{\perp}\cdot w_n \Vert_{\dot{H}^{-5}}^2 & \lesssim \sum_{|j| \simeq \lambda_{q+1}} |j|^{-10} (\lambda_{q+1}\lambda_n)^2\\ &\lesssim \lambda_{q+1}^{-7} \lambda_n^2.
    \end{split}
\end{equation*}
Thus, for $\lambda_{q+1}$ large enough we obtain
\begin{equation}\label{eq:ind_est_1}
\begin{split}
    \sum_{n \leq q} \Vert \Lambda w_{q+1} \nabla^{\perp}\cdot w_n \Vert_{\dot{H}^{-5}} &\lesssim \lambda_{q+1}^{-7/2} \sum_{n \leq q} \lambda_n\\
    &< 2^{-2q-100}.
    \end{split}
\end{equation}
Using the same argument, we may deduce that
\begin{equation}\label{eq:ind_est_2}
\begin{split}
    \sum_{m \leq q} \Vert \Lambda w_{m} \nabla^{\perp}\cdot w_{q+1} \Vert_{\dot{H}^{-5}} &\lesssim \lambda_{q+1}^{-7/2} \sum_{m \leq q} \lambda_m \\
    &< 2^{-2q-100}.
    \end{split}
\end{equation}
From our inductive hypothesis, ~\eqref{eq:ind_est_1}, and ~\eqref{eq:ind_est_2} we have
$$
    \sum_{\substack{n,m \leq q+1\\n\not=m}} \Vert \Lambda w_m \nabla^{\perp}\cdot w_n \Vert_{\dot{H}^{-5}} \leq C_1 - 2^{-q} + 2^{-2q-99} < C_1 - 2^{-q-1}.
$$
We now turn to the diagonal terms. Recall from our inductive assumption we assume
$$
    \sum_{n \leq q} \Vert \Lambda w_n \nabla^{\perp}\cdot w_n \Vert_{\dot{H}^{-5}} < C_2 - 2^{-q+100}.
$$
We have
\begin{equation}\label{eq:diag_exp}
    \sum_{n \leq q+1} \Vert \Lambda w_n \nabla^{\perp}\cdot w_n \Vert_{\dot{H}^{-5}} = \sum_{n \leq q} \Vert \Lambda w_n \nabla^{\perp}\cdot w_n \Vert_{\dot{H}^{-5}} + \Vert \Lambda w_{q+1} \nabla^{\perp}\cdot w_{q+1} \Vert_{\dot{H}^{-5}}
\end{equation}
and so it remains to estimate $\Vert \Lambda w_{q+1} \nabla^{\perp}\cdot w_{q+1} \Vert_{\dot{H}^{-5}}$. One can check that for $u:\T^2 \to \R^2$ smooth and divergence free, one has
$$
    \Lambda u = R^{\perp}(\nabla^{\perp} \cdot u).
$$
Then if $T$ denotes the rotation by $\pi/2$, then
$$
    \Lambda w_{q+1} \nabla^{\perp}\cdot w_{q+1} = T(\Lambda w_{q+1}^{\perp} \nabla^{\perp}\cdot w_{q+1}) = T(R(\nabla^{\perp} \cdot w_{q+1}) \nabla^{\perp} \cdot w_{q+1}).
$$
Hence using that the Fourier transform of a rotation is the rotation of the Fourier transform, we get
\begin{equation*}
    \begin{split}
        \Vert \Lambda w_{q+1} \nabla^{\perp}\cdot w_{q+1} \Vert_{\dot{H}^{-5}}^2 & = \sum_{j \not =0} |j|^{-10} \left|\left(T(R(\nabla^{\perp} \cdot w_{q+1}) \nabla^{\perp} \cdot w_{q+1})\right)^{\wedge}(j)\right|^2 \\
        & = \sum_{j \not =0} |j|^{-10} \left|\left(R(\nabla^{\perp} \cdot w_{q+1}) \nabla^{\perp} \cdot w_{q+1}\right)^{\wedge}
        (j)\right|^2.
    \end{split}
\end{equation*}
We recall these terms $R(\nabla^{\perp} \cdot w_{q+1}) \nabla^{\perp} \cdot w_{q+1}$ are precisely the terms which were treated in Section \ref{section_oscillation_error}. The only difference is there is no inverse divergence and we need to manually add and subtract away the pressure terms which were only subtracted before. The lack of the inverse divergence operator being present is the reason the regularity has to be lowered by $1$. From the computations in Section \ref{section_oscillation_error}, we identified two pressure terms for each $k \in \Omega$ which from ~\eqref{eq:pressure_decomp} were
\begin{equation*}
    p_{q+1,k,1} = -\frac 12 \Lambda^{-1} \vartheta_{q+1,k}\vartheta_{q+1,-k}
\end{equation*}
and
\begin{equation*}
    p_{q+1,k,2} = -\frac{\tilde{C}\lambda_{q+1}}{2} a^2_k(R_q).
\end{equation*}
We aim to estimate each of these terms in $\dot{H}^{-4}$ norm and demonstrate they are bounded by some constant multiple of $\Vert R_q \Vert_{\dot{H}^{-4}}$. We start with $p_{q+1,k,2}$. Recall from Lemma \ref{lem:geom} we chose
$$
    \Omega = \{\pm e_1, \pm (3/5,4/5), \pm (3/5,-4/5)\}.
$$
Put $k_1 = e_1$, $k_2 = (3/5,4/5)$, and $k_3 = (3/5,-4/5)$. Then since $a_k(R_q) = a_{-k}(R_q)$ and $k^{\perp} \otimes k^{\perp} = (-k)^{\perp} \otimes (-k)^{\perp}$ then we have that
\begin{equation}\label{eq:lin_combo}
    \begin{split}
        \frac{I - C\epsilon_q R_q}{\lambda_{q+1}\epsilon_q} & = \frac{1}{2}\sum_{k \in \Omega} a^2_k(R_q) k^{\perp} \otimes k^{\perp}\\
        & = a^2_{k_1}(R_q) k_1^{\perp} \otimes k_1^{\perp} + a^2_{k_2}(R_q) k_2^{\perp} \otimes k_2^{\perp} + a^2_{k_3}(R_q) k_3^{\perp} \otimes k_3^{\perp}.
    \end{split}
\end{equation}
Since $k_1^{\perp} \otimes k_1^{\perp}$, $k_2^{\perp} \otimes k_2^{\perp}$ and $k_3^{\perp} \otimes k_3^{\perp}$ form a basis of the space of $2 \times 2$ symmetric matrices, then we may solve for each $a^2_k(R_q)$ in terms of the elements of $\Omega$ and $I - C\epsilon_qR_q$. Indeed, we have using ~\eqref{eq:lin_combo} that
$$
    a^2_{k_1}(R_q) = \left(\frac{I - C\epsilon_q R_q}{\lambda_{q+1}\epsilon_q}\right)_{22} - \frac{9}{16} \left(\frac{I - C\epsilon_q R_q}{\lambda_{q+1}\epsilon_q}\right)_{11},
$$
$$
    a^2_{k_2}(R_q) = \frac{25}{32}\left(\frac{I - C\epsilon_q R_q}{\lambda_{q+1}\epsilon_q}\right)_{11} - \frac{25}{24} \left(\frac{I - C\epsilon_q R_q}{\lambda_{q+1}\epsilon_q}\right)_{12},
$$
and
$$
    a^2_{k_3}(R_q) = \frac{25}{32}\left(\frac{I - C\epsilon_q R_q}{\lambda_{q+1}\epsilon_q}\right)_{11} + \frac{25}{24} \left(\frac{I - C\epsilon_q R_q}{\lambda_{q+1}\epsilon_q}\right)_{12}.
$$
And so
\begin{equation*}
    \begin{split}
        \Vert a^2_{k_1} \Vert_{\dot{H}^{-4}} & \leq \left\Vert\left(\frac{I - C\epsilon_q R_q}{\lambda_{q+1}\epsilon_q}\right)_{22}\right\Vert_{\dot{H}^{-4}} + \frac{9}{16} \left\Vert \left(\frac{I - C\epsilon_q R_q}{\lambda_{q+1}\epsilon_q}\right)_{11} \right\Vert_{\dot{H}^{-4}} \\
                                             & \leq \frac{25}{16} \left\Vert \frac{I - C\epsilon_q R_q}{\lambda_{q+1}\epsilon_q}\right\Vert_{\dot{H}^{-4}}                                                                                                                               \\
                                             & \leq \frac{4C}{\lambda_{q+1}} \Vert R_q \Vert_{\dot{H}^{-4}}.
    \end{split}
\end{equation*}
Similarly we get
\begin{equation*}
    \Vert a^2_{k_2}(R_q) \Vert_{\dot{H}^{-4}} \leq \frac{4C}{\lambda_{q+1}} \Vert R_q \Vert_{\dot{H}^{-4}}
\end{equation*}
and
\begin{equation*}
    \Vert a^2_{k_3}(R_q) \Vert_{\dot{H}^{-4}} \leq \frac{4C}{\lambda_{q+1}} \Vert R_q \Vert_{\dot{H}^{-4}}.
\end{equation*}
Thus for all $k \in \Omega$ one has
\begin{equation}\label{eq:pk2_est}
    \begin{split}
        \Vert p_{q+1,k,2} \Vert_{\dot{H}^{-4}} \leq \frac{\tilde{C}\lambda_{q+1}}{2} \left\Vert a^2_k(R_q) \right\Vert_{\dot{H}^{-4}} \leq 2C\tilde{C} \Vert R_q \Vert_{\dot{H}^{-4}} < 2\Vert R_q \Vert_{\dot{H}^{-4}}.
    \end{split}
\end{equation}
Note above we utilize the fact that $C\tilde{C} < 1$. It is easy to see that
$$
    \left(x_2 + \frac{1}{4}\right)^2 - 2\left(x_1 + \frac{1}{8}\right)^2 > 0 \quad \text{for} \quad |x| \leq \frac{1}{40}
$$
which implies $C^{-1} - \tilde{C} > 0$. ~\eqref{eq:pk2_est} follows from this observation.\\
The boundedness of $p_{q+1,k,1}$ follows almost exactly the same procedure as the one in Section \ref{section_oscillation_error}. We sketch the details focusing on the differences. Since
$$
    \left(\Lambda^{-1} \vartheta^{\pm}_{q+1,k}\right)^{\wedge}(\xi) = -\frac{|\xi|}{\sigma_{q+1}} \hat{K}_{\simeq 1}\left(\frac{\xi -(k\pm 2k^{\perp})\sigma_{q+1}}{\lambda_{q+1}} \right) \left(a_k(R_q) \rho^{k}_{q+1}\right)^{\wedge}(\xi - (k \pm 2k^{\perp})\sigma_{q+1})
$$
and
$$
    \left( \vartheta^{\pm}_{q+1,-k}\right)^{\wedge}(\eta) = \frac{2\pi|\eta|^2}{\sigma_{q+1}} \hat{K}_{\simeq 1}\left(\frac{\eta+(k\pm 2k^{\perp})\sigma_{q+1}}{\lambda_{q+1}} \right) \left(a_k(R_q) \rho^{k}_{q+1}\right)^{\wedge}(\eta + (k \pm 2k^{\perp}) \sigma_{q+1}).
$$
Each of the four sign combinations has the same structure as to what we have already seen in ~\eqref{eq:S^m_decomp}; when the signs are the same we expect the terms to be of low frequency and when they differ, they are of high frequency. So, employing precisely the same argument we used to show ~\eqref{eq:high_freq_ests}, we may also deduce that
\begin{equation}\label{eq:low_freq_grad_est}
    \Vert \Lambda^{-1} \vartheta^{+}_{q+1,k} \vartheta^{-}_{q+1,-k} \Vert_{\dot{H}^{-4}} + \Vert \Lambda^{-1} \vartheta^{-}_{q+1,k} \vartheta^{+}_{q+1,-k} \Vert_{\dot{H}^{-4}} < \Vert R_q \Vert_{\dot{H}^{-4}}.
\end{equation}
Now we proceed with the argument for the low frequency terms. As we saw before, the argument for the $+$ and $-$ cases is identical and can be handled simultaneously. For ease of notation, we simply treat the case where both signs are $+$. So we have
\begin{equation*}
    \begin{split}
        \Lambda^{-1} \vartheta^{+}_{q+1,k}\vartheta^{+}_{q+1,-k} &= \int_{\R^2} \int_{\R^2} M^+_{q+1,k} (\xi,\eta) \left(a_k(R_q) \rho^{k}_{q+1}\right)^{\wedge}(\xi)\\
        &\times \left(a_k(R_q) \rho^{k}_{q+1}\right)^{\wedge}(\eta) e^{2\pi i (\xi + \eta) \cdot x}\, d\xi\, d\eta
    \end{split}
\end{equation*}
where, after the change of variables
\[
    \xi \mapsto \xi + (k+2k^{\perp})\sigma_{q+1}, \quad \eta \mapsto \eta -(k+2k^{\perp})\sigma_{q+1},
\]
the symbol takes the form
$$
    M^+_{q+1,k} (\xi,\eta) = -2\pi\frac{|\xi + \sigma_{q+1}(k +2 k^{\perp})| | \eta - \sigma_{q+1}(k + 2k^{\perp})|^2}{\sigma_{q+1}^2} \hat{K}_{\simeq 1}\left(\frac{\xi}{\lambda_{q+1}}\right) \hat{K}_{\simeq 1}\left(\frac{\eta}{\lambda_{q+1}}\right).
$$
Defining the rescaled symbol
\begin{equation}\label{eq:M^*_2}
    M^{+*}_k(\xi,\eta) = -2\pi |\xi + (k + 2k^\perp)| | \eta - (k + 2k^\perp)|^2 \hat{K}_{\simeq 1}\left(\frac{5}{8} \xi\right) \hat{K}_{\simeq 1}\left(\frac{5}{8} \eta\right),
\end{equation}
we see that $M^{+*}_k$ is compactly supported and smooth. If we put
\begin{equation}\label{eq:kern_2}
    K^+_{q+1,k}(y,z) = \sigma_{q+1}^5 \left(M^{+*}_k\right)^{\vee}(\sigma_{q+1}y,\sigma_{q+1}z),
\end{equation}
then observe that $K_{q+1,k}$ satisfies the same bounds as ~\eqref{eq:kernel_est} and
$$
    \Lambda^{-1} \vartheta^+_{q+1,k}\vartheta^+_{q+1,-k} = \int_{\R^2} \int_{\R^2} K^+_{q+1,k}(y,z) a_k(x-y)\rho_{q+1}^k(x-y) a_k(x-z) \rho_{q+1}^k(x-z)\, dy\, dz.
$$
Notice, we have again suppressed the dependence of $a_k$ on the matrix $R_q$. Splitting $\rho_{q+1}^k(x-y)\rho_{q+1}^k(x-z)$ into its mean and off-mean component, we arrive at the decomposition
$$
    \Lambda^{-1} \vartheta^+_{q+1,k}\vartheta^+_{q+1,-k} = I_1 + I_2
$$
where
\begin{equation*}
    I_1 = \int_{\R^2} \int_{\R^2} K^+_{q+1,k}(y,z) \mathbb{P}_{\not =0,x}\left(\rho_{q+1}^k(x-y) \rho_{q+1}^k(x-z)\right) a_k(x-y) a_k(x-z)\, dy\, dz
\end{equation*}
and
\begin{equation*}
    I_2 = \int_{\R^2} \int_{\R^2} K^+_{q+1,k}(y,z) \mathbb{P}_{ =0,x}\left(\rho_{q+1}^k(x-y) \rho_{q+1}^k(x-z)\right) a_k(x-y) a_k(x-z)\, dy\, dz
\end{equation*}
where again the $x$ subscript indicates the mean is taken with respect to the $x$ variable. Following precisely the same argument for ~\eqref{eq:Q2_est}, we may show that
\begin{equation}\label{eq:I1_est}
    \Vert I_1 \Vert_{\dot{H}^{-4}} < \Vert R_q \Vert_{\dot{H}^{-4}}
\end{equation}
for $\lambda_{q+1}$ chosen large enough. For $I_2$, we again write
\begin{equation}\label{eq:I2_decomp}
    \begin{split}
        I_2 & = a_k^2(x) \int_{\R^2} \int_{\R^2} K^+_{q+1,k}(y,z) \mathbb{P}_{ =0,x}\left(\rho_{q+1}^k(x-y) \rho_{q+1}^k(x-z)\right)\, dy\, dz                                      \\
            & - a_k(x)\int_0^1 \int_{\R^2} \int_{\R^2} K^+_{q+1,k}(y,z) \mathbb{P}_{ =0,x}\left(\rho_{q+1}^k(x-y) \rho_{q+1}^k(x-z)\right) y\cdot \nabla a_k(x-ty) \, dy\, dz\, dt\ \\
            & - a_k(x)\int_0^1 \int_{\R^2} \int_{\R^2} K^+_{q+1,k}(y,z) \mathbb{P}_{ =0,x}\left(\rho_{q+1}^k(x-y) \rho_{q+1}^k(x-z)\right) z\cdot \nabla a_k(x-tz)\, dy\, dz\, dt   \\
            & + \int_{\R^2} \int_{\R^2} K^+_{q+1,k}(y,z) \mathbb{P}_{ =0,x}\left(\rho_{q+1}^k(x-y) \rho_{q+1}^k(x-z)\right)                                                         \\
            & \times \prod_{w \in \{y,z\}} \left(w \cdot \int_0^1 \nabla a_k(x-tw)\, dt\right) \, dy\, dz.
    \end{split}
\end{equation}
Using the same arguments used to bound the final three expressions in ~\eqref{eq:Q^1_decomp}, one can show the final three terms in ~\eqref{eq:I2_decomp} can be bounded in $\dot{H}^{-4}$ norm by $\Vert R_q \Vert_{\dot{H}^{-4}}$ for $\lambda_{q+1}$ chosen large enough. For the first term in ~\eqref{eq:I2_decomp}, which we refer to as $I_3,$ using ~\eqref{eq:mean_prod} and properties of the Fourier transform, one has
\begin{equation*}
    \begin{split}
        I_3 & = \sum_{n \in \Z^2} \lambda_{q+1}^{2(\epsilon-1)} \left|\hat{\phi}\left(\lambda_{q+1}^{\epsilon-1} n\right)\right|^2 a_k^2(x)\int_{\R^2} \int_{\R^2} K^+_{q+1,k}(y,z) e^{-2\pi i 5 \lambda_{q+1}^{\epsilon}(n_1 k + n_2 k^{\perp})\cdot (y-z)}\, dy\, dz \\
            & = \sum_{n \in \Z^2} \lambda_{q+1}^{2(\epsilon-1)} \left|\hat{\phi}\left(\lambda_{q+1}^{\epsilon-1} n\right)\right|^2 a_k^2(x) \hat{K}^+_{q+1,k}\left(5\lambda_{q+1}^\epsilon \mathcal{O}_kn,-5\lambda_{q+1}^\epsilon \mathcal{O}_kn\right).
    \end{split}
\end{equation*}
From ~\eqref{eq:M^*_2} and ~\eqref{eq:kern_2} we see that
$$
    \hat{K}_{q+1,k}\left(-\lambda_{q+1}^{\epsilon}5nk, \lambda_{q+1}^{\epsilon}5nk\right) = -2\pi\frac{\left|5\lambda_{q+1}^{\epsilon}\mathcal{O}_kn - \sigma_{q+1}(k+2k^\perp)\right|^3}{\sigma_{q+1}^2} \left|\hat{K}_{\simeq 1}\left(5\lambda_{q+1}^{\epsilon-1}\mathcal{O}_kn\right)\right|^2.
$$
And thus
\begin{equation*}
    \begin{split}
        I_3 &= a_k^2(x)\sum_{n \in \Z^2} -2\pi \lambda_{q+1}^{2(\epsilon-1)} \left|\hat{\phi}\left(\lambda_{q+1}^{\epsilon-1} n\right)\right|^2\\
        &\times \frac{\left|5\lambda_{q+1}^{\epsilon}\mathcal{O}_kn - \sigma_{q+1}(k+2k^\perp)\right|^3}{\sigma_{q+1}^2} \left|\hat{K}_{\simeq 1}\left(5\lambda_{q+1}^{\epsilon-1}\mathcal{O}_kn\right)\right|^2.
    \end{split}
\end{equation*}
Since $|5\lambda_{q+1}^{\epsilon-1}\mathcal{O}_kn| \leq 1/40$, using crude estimates we see
\begin{equation*}
    \begin{split}
        \frac{\left|5\lambda_{q+1}^{\epsilon}\mathcal{O}_kn - \sigma_{q+1}(k+2k^\perp)\right|^3}{\sigma_{q+1}^2} \left|\hat{K}_{\simeq 1}\left(5\lambda_{q+1}^{\epsilon-1}\mathcal{O}_kn\right)\right|^2 & \leq 320\lambda_{q+1}\left(\frac{1}{40} + \frac{\sqrt{5}}{8}\right)^3 \\
                                                                 & < 16\lambda_{q+1}.
    \end{split}
\end{equation*}
Using the same arguments as in ~\eqref{eq:Q11_final_est} and ~\eqref{eq:pk2_est} one can show that $\Vert I_3 \Vert_{\dot{H}^{-4}} < 128 \Vert R_q \Vert_{\dot{H}^{-4}}$. Hence
\begin{equation}\label{eq:I2_est}
    \|I_2\|_{\dot{H}^{-4}} < 256\Vert R_q \Vert_{\dot{H}^{-4}}
\end{equation}
for $\lambda_{q+1}$ chosen large enough. Combining ~\eqref{eq:low_freq_grad_est}, ~\eqref{eq:I1_est}, and ~\eqref{eq:I2_est} gives
\begin{equation}\label{eq:pk1_est}
    \Vert p_{q+1,k,1} \Vert_{\dot{H}^{-4}} < 1024 \Vert R_q \Vert_{\dot{H}^{-4}}.
\end{equation}
Now, from ~\eqref{eq:pk1_est} we get
\begin{equation*}
    \begin{split}
        \Vert \Lambda w_{q+1} \nabla^{\perp}\cdot w_{q+1} \Vert_{\dot{H}^{-5}} & = \left\Vert \sum_{k \in \Omega} \left(\nabla p_{q+1,k,1} + \operatorname{div}\left(S(\Lambda^{-1} \vartheta_{q+1,k}, R\vartheta_{q+1,-k})\right)\right)\right\Vert_{\dot{H}^{-5}}\\
        & \leq \sum_{k \in \Omega} \left\Vert p_{q+1,k,1} \right\Vert_{\dot{H}^{-4}} +  \left\Vert \sum_{k\in \Omega} S(\Lambda^{-1} \vartheta_{q+1,k}, R\vartheta_{q+1,-k})\right\Vert_{\dot{H}^{-4}}\\
           & < 8192\Vert R_q \Vert_{\dot{H}^{-4}} +  \left\Vert \sum_{k \in \Omega} \left(S(\Lambda^{-1} \vartheta_{q+1,k}, R\vartheta_{q+1,-k}) - p_{q+1,k,2}I\right) \right\Vert_{\dot{H}^{-4}}\\
        & + \sum_{k \in \Omega} \Vert p_{q+1,k,2}\Vert_{\dot{H}^{-4}}.
    \end{split}
\end{equation*}
Then, since $|\Omega| = 6 < 8$, from ~\eqref{eq:RO_est}, ~\eqref{eq:pk2_est} we have
\begin{equation}\label{eq:quant_bound}
\begin{split}   
    \Vert \Lambda w_{q+1} \nabla^{\perp}\cdot w_{q+1} \Vert_{\dot{H}^{-5}} &< 8192\Vert R_q \Vert_{\dot{H}^{-4}} +  2\left\Vert R_q \right\Vert_{\dot{H}^{-4}} +16\Vert R_q\Vert_{\dot{H}^{-4}}\\& = 8210 \Vert R_q \Vert_{\dot{H}^{-4}}\\& < 2^{-q+15}.
    \end{split}
\end{equation}
So using ~\eqref{eq:diag_exp}, ~\eqref{eq:quant_bound}, and our inductive hypothesis we have that
\begin{equation*}
    \sum_{n \leq q+1} \Vert \Lambda w_n \nabla^{\perp}\cdot w_n \Vert_{\dot{H}^{-5}} < C_2 - 2^{-q+100} + 2^{-q+15} < C_2 - 2^{-q+99}.
\end{equation*}
This proves \ref{i:7}.

\subsection{Proof of Item \ref{i:8}}\label{Sec:Regularity}

We will require the following variant of Bernstein's lemma.

\begin{lemma}\label{lem:Lp_Bernstein}
    Recall the projection operators $\mathbb{P}_{2^j}$ from Definition \ref{def:projs}. For $1 \leq p \leq \infty$ and $u$ smooth we have
    \begin{equation}\label{eq:Lp_Bernstein}
        \Vert \Lambda^s \mathbb{P}_{2^j}(u) \Vert_{L^p} \lesssim 2^{sj} \Vert u \Vert_{L^p}.
    \end{equation}
\end{lemma}
\begin{proof}
    Recall
    $$
        \mathbb{P}_{2^j}(u) = u \ast \left(\sum_{k \in \Z^2} \phi_j(k)e^{2\pi i k \cdot x}\right)
    $$
    Thus
    $$
        \left(\Lambda^s \mathbb{P}_{2^j}(u)\right)^{\wedge}(k) = \left(2\pi |k|\right)^s \phi_j(k) \hat{u}(k),
    $$
    and so
    \begin{equation}\label{eq:Bernstein_conv}
        \Lambda^s \mathbb{P}_{2^j}(u) = u \ast \left(\sum_{k \in \Z^2} (2\pi |k|)^s \phi_j(k)e^{2\pi i k \cdot x}\right).
    \end{equation}
    Utilizing Poisson summation one can show that
    $$
        \left\Vert\sum_{k \in \Z^2} (2\pi |k|)^s \phi_j(k)e^{2\pi i k \cdot x}\right\Vert_{L^1(\T^2)} \lesssim 2^{sj}.
    $$
    So applying this as well as Young's convolution inequality to ~\eqref{eq:Bernstein_conv} gives ~\eqref{eq:Lp_Bernstein}.
\end{proof}
\noindent
Now we proceed with the rest of the argument. Assume $-1/2 \leq \alpha < 1/2$ and fix $1 \leq p < 2$. From \ref{i:6} we have that $\mathbb{P}_{\lambda_{q+1}}(w_{q+1}) = w_{q+1}$. Hence applying this as well as Lemma \ref{lem:Bernstein} we obtain
\begin{equation*}
    \begin{split}
        \Vert \Lambda^{1/2}(v_{q+1} - v_q)\Vert_{L^p} & = \Vert \Lambda^{1/2}w_{q+1} \Vert_{L^p}                                                        \\
                                                      & \lesssim \lambda_{q+1}^{1/2} \Vert w_{q+1} \Vert_{L^p}                                          \\
                                                      & \lesssim \lambda_{q+1}^{1/2} \lambda_{q+1}^{-1/2} \lambda_{q+1}^{(1-\epsilon)(1 - \frac{2}{p})} \\
                                                      & = \lambda_{q+1}^{(1-\epsilon)(1 - \frac{2}{p})}.
    \end{split}
\end{equation*}
Since $\epsilon < 1$, we may choose $\lambda_{q+1}$ large enough to give
\begin{equation*}
    \Vert \Lambda^{1/2}(v_{q+1} - v_q)\Vert_{L^p} < 2^{-q-1}.
\end{equation*}
Similarly, fix $\alpha < -1/2$ and $\epsilon'$ and $p$ satisfying ~\eqref{eq:conds}. Then we have
\begin{equation*}
    \begin{split}
        \left\Vert \Lambda^{3/2-\epsilon'}(v_{q+1} - v_q)\right\Vert_{L^p} & = \left\Vert \Lambda^{3/2-\epsilon'} w_{q+1}\right\Vert_{L^p}\\
        &\lesssim \lambda_{q+1}^{3/2-\epsilon'} \lambda_{q+1}^{-1/2} \lambda_{q+1}^{(1-\epsilon)(1 - \frac{2}{p})} \\
        & < \lambda_{q+1}^{1-\epsilon' + (1-\epsilon)(1-\frac{2}{p})}.
    \end{split}
\end{equation*}
Using ~\eqref{eq:conds}, we see
$$
\left(1-\epsilon\right)\left(1-\frac{2}{p}\right) < \left(1-\epsilon'\right)\frac{p}{2-p}\left(1 - \frac{2}{p}\right) = \epsilon' - 1
$$
thus
$$
1-\epsilon' + (1 - \epsilon)\left(1 - \frac{2}{p}\right) < 0
$$
and so, we may choose $\lambda_{q+1}$ large enough to ensure
\begin{equation*}
    \Vert \Lambda^{3/2-\epsilon'}(v_{q+1} - v_q)\Vert_{L^p} < 2^{-q-1}.
\end{equation*}
This completes the proof.

\noindent\textsc{Department of Mathematics, Purdue University, West Lafayette, IN, USA.}
\vspace{.03in}
\newline\noindent\textit{Email address}: \href{mailto:ngismond@purdue.edu}{ngismond@purdue.edu}.\\
\noindent\textsc{"Simion Stoilow" Institute of Mathematics of the Romanian Academy, Calea Grivitei Street, no. 21, 010702 Bucharest, Romania.}
\newline\noindent\textit{Email address}: \href{mailto:sasharadu@icloud.com}{sasharadu@icloud.com}.

\end{document}